\documentclass{amsart}

\input xypic
 \xyoption{all}

\usepackage{bbm}
\usepackage{empheq}
\usepackage{color}
\usepackage[colorlinks,linkcolor=blue,anchorcolor=blue,citecolor=blue]{hyperref}
\usepackage{amssymb,amstext, amsbsy, amscd}
\usepackage[mathscr]{eucal}
\usepackage{times}
\usepackage{amsmath, amsthm, amsfonts,extarrows}

\usepackage[all,cmtip]{xy}

\newtheorem{thm}{Theorem}[section]
\newtheorem{prop}[thm]{Proposition}
\newtheorem{lemma}[thm]{Lemma}
\newtheorem{cor}[thm]{Corollary}
\newtheorem{defn}[thm]{Definition}%[section]
\newtheorem{example}[thm]{Example}
\newtheorem{remark}[thm]{Remark}

\newtheorem{conjecture}[thm]{Conjecture}

\newtheorem{prop-defn}[thm]{Proposition/Definition}

 \newcommand{\ba}{\begin{eqnarray}}
   \newcommand{\na}{\end{eqnarray}}
   \newcommand{\ban}{\begin{eqnarray*}}
   \newcommand{\nan}{\end{eqnarray*}}

\newcommand{\bC}{{\mathbb C}}

\newcommand{\bP}{{\mathbb P}}

\newcommand{\bR}{{\mathbb R}}
\newcommand{\bZ}{{\mathbb Z}}

\newcommand{\cO}{{\mathcal O}}

\newcommand{\tspec}{\textrm{Spec}}

  \newcommand{\<}{\langle}
  \renewcommand{\>}{\rangle}

\newcommand{\suml}{\sum\limits}
\newcommand{\prodl}{\prod\limits}

\newtheorem*{conjO}{Conjecture $\cO$}
\newtheorem*{GammaI}{Gamma conjecture I}

\begin{document}{\allowdisplaybreaks[4]

\title{Gamma conjecture I for del Pezzo surfaces}

\author{Jianxun Hu }
\address{School of Mathematics, Sun Yat-sen University, Guangzhou 510275, P.R. China}

%\curraddr{Department of Mathematics }
\email{stsjxhu@mail.sysu.edu.cn}
\thanks{ %The first author  is supported in part by %a RGC research grant from the Hong Kong Government.
%The second author  is  supported in part %by KRF-2007-341-C00006.
 }

\author{Huazhong Ke}
\address{School of Mathematics, Sun Yat-sen University, Guangzhou 510275, P.R. China}
%\curraddr{Department of Mathematics }
\email{kehuazh@mail.sysu.edu.cn}
\thanks{ %The first author  is supported in part by %a RGC research grant from the Hong Kong Government.
%The second author  is  supported in part %by KRF-2007-341-C00006.
 }

\author{Changzheng Li}
 \address{School of Mathematics, Sun Yat-sen University, Guangzhou 510275, P.R. China}
\email{lichangzh@mail.sysu.edu.cn}

\author{Tuo Yang }
\address{Carrikk School of Management, Boston College, 140 Commonwealth Avenue, Chestnut Hill, MA 02467}
%\curraddr{Department of Mathematics }
\email{yanganj@bc.edu}
\thanks{ %The first author  is supported in part by %a RGC research grant from the Hong Kong Government.
%The second author  is  supported in part %by KRF-2007-341-C00006.
 }

\date{%January 1, 2001 and, in revised form, June 22, 2001.
      }



%\subtitle{}

\begin{abstract}
    Gamma conjecture I and the underlying Conjecture $\mathcal{O}$   for Fano manifolds were proposed by Galkin, Golyshev and Iritani recently.   We show that both conjectures  hold for all two-dimensional Fano manifolds. We prove Conjecture $\mathcal{O}$  by deriving a generalized Perron-Frobenius theorem on eigenvalues of real matrices and   a vanishing result of certain Gromov-Witten invariants for del Pezzo surfaces. We prove Gamma conjecture I by applying mirror techniques proposed by Galkin-Iritani together with the study of Gamma conjecture I for weighted projective spaces.
%%We also introduce a particular anti-canonical divisor of $F\ell_{n_1, n_2, \cdots, n_k; n}$, which degenerates to the anti-canonical divisor of the Gelfand-Cetlin toric variety.
  \end{abstract}

\maketitle
 % \tableofcontents

\section{Introduction}

%In their paper \cite{GGI}, Galkin, Golyshev and Iritani proposed the so-called  Gamma conjecture  I and the underlying conjecture $\cO$  for any %Fano manifold $X$, namely
% for  a compact complex manifold whose anti-canonical
%line bundle is ample.  Conjecture  $\mathcal{O}$ concerns with eigenvalues of a linear operator on the small quantum cohomology ring $QH^*(X)$,
%and Gamma conjecture I cares about the asymptotic expansion of Givental's small $J$-function $J_X(t)$ by assuming  conjecture $\mathcal{O}$ first.
%Here both $QH^*(X)$ and $J_X(t)$ are basic objects in the study  of Gromov-Witten theory of $X$.
%Actually, Gamma conjecture II was also proposed in \cite{GGI} as a refinement of a part of the original Dubrovin conjecture \cite{Du1998}, %independent of the new formulation in \cite{Du2013}.
%Nevertheless, we will focus on the former two conjectures in the present paper.

{%\color{red}
Let $X$ be a Fano manifold, namely a compact manifold whose anti-canonical line bundle is ample. Denote by $QH^*(X)$ the small quantum cohomology of $X$ and $J_X(t)$  Givental's small $J$-function in Gromov-Witten theory of $X$. In \cite{GGI}, Galkin, Golyshev and Iritani proposed the so-called  Gamma conjectures and the underlying conjecture $\cO$  for any Fano manifold $X$. Conjecture  $\mathcal{O}$ concerns with eigenvalues of a linear operator on the small quantum cohomology ring $QH^*(X)$. Gamma conjecture consists of Gamma conjecture I and II. Gamma conjecture I cares about the asymptotic expansion of Givental's small $J$-function $J_X(t)$ by assuming  conjecture $\mathcal{O}$ first. Gamma conjecture II was also proposed in \cite{GGI} as a refinement of a part of the original Dubrovin conjecture \cite{Du1998}, independent of the new formulation in \cite{Du2013, CDG}. In this paper, we will focus on Conjecture  $\mathcal{O}$ and Gamma conjecture I.

}

 To be more precise, the quantum cohomology ring $QH^*(X)$ is a  deformation
 of the classical cohomology ring $H^*(X)=H^*(X,\mathbb{Q})$ by incorporating genus zero, three-point Gromov-Witten invariants of $X$. As vector spaces,
 we have  $QH^*(X)=H^*(X)\otimes \mathbb{Q}[q_1,\cdots, q_m]$, where  $m$ is the second Betti number of $X$ and $q_i$'s are quantum variables parameterizing a basis of effective curve classes in $H_2(X, \mathbb{Z})$.
 The quantum multiplication by the first Chern class $c_1(X)$ of $X$ induces a linear operator $\hat c_1=c_1(X)\star_{\mathbf{q}=\mathbf{1}}$ on the even part  $H^\bullet(X):=H^{\rm ev}(X)=QH^{\rm ev}(X)|_{\mathbf{q}=(1,\cdots, 1)}$, which is a vector space of finite dimension.
 The  so-called \textit{Property $\mathcal{O}$}, introduced for general Fano manifolds,    consists of two parts. The first half says that the spectral radius $\rho(\hat c_1)=\max\{|\lambda|: \lambda\mbox{ is an eigenvalue of } \hat c_1\}$ itself is a eigenvalue of $\hat c_1$ of multiplicity one. {%\color{red}
 The second half says that $|\lambda|\neq \rho(\hat c_1)$ whenever $\lambda\neq \rho(\hat c_1)$ is an eigenvalue of $\hat c_1$ in the case of   Fano manifolds  of Fano index one (see section 2.1 for  general cases).}
Galkin, Golyshev and Iritani \cite{GGI} made the following conjecture with the name $\mathcal{O}$ indicating a deep relation with homological mirror symmetry.
\begin{conjO}
   Every Fano manifold satisfies   Property $\mathcal{O}$.
\end{conjO}
 \noindent There have been complete classifications of Fano manifolds of small dimension.  Any one-dimensional Fano manifold is
isomorphic to the complex projective line $\bP^1$. Every two-dimensional Fano manifold, i.e.
a del Pezzo surface,   is either isomorphic to $\bP^1\times \bP^1$ or the blowup $X_r$ of $\bP^2$ at $r$ points in
general position $(0 \leq r\leq 8)$. As one main result of the present paper, we prove  conjecture $\mathcal{O}$ for $X_r (1\leq r\leq 8)$ (which are all of Fano index one), in addition to the known cases $\bP^2$ and $\bP^1\times \bP^1$. Namely, we show the following.
 \begin{thm}\label{thmconjO}
 Every del Pezzo surface satisfies Property  $\cO$.
\end{thm}

The formulation of Gamma conjecture I requires a bit more knowledge on the Gromov-Witten theory.
 One one hand, the restriction $J_X(t)$ of Givental's big $J$-function to the anti-canonical line is a (multi-valued) $H^*(X)$-valued function over $\mathbb{C}^*$ defined by incorporating genus zero, one-point Gromov-Witten invariants with gravitational   descendents.
 Under the assumption that $X$ satisfies conjecture $\mathcal{O}$, $J_X(t)$ admits an asymptotic expansion at the infinity, whose leading term $A_X \in H^\bullet(X)$ is unique up to a nonzero scalar and is called the \textit{principal asymptotic class} of $X$.
 On the other hand, the Gamma class $\hat \Gamma_X=\prod_{i=1}^{\dim X}\Gamma(1+x_i)$
    \cite{HKTY, Libg, Lu, Iri}  of $X$ is a  real characteristic class, defined by   the Chern root $x_1,\cdots, x_{\dim X}$ of the tangent bundle $TX$ of $X$ and
  Euler's $\Gamma$-function  $\Gamma(z)=\int_0^\infty e^tt^{z-1}dt$.
  It has the following expansion (see e.g. \cite{GGI}):
  $$\hat \Gamma_X =\exp\big(-C_{\rm eu}c_1(X)+\sum_{k=2}^\infty (-1)^k(k-1)!\zeta(k)\mbox{ch}_k(TX)\big)  \in H^\bullet (X, \mathbb{R})$$
where $C_{\rm eu}=0.5772156...$ is the Euler-Mascheroni   constant, $\zeta(k)=\sum_{n=1}^\infty{1\over n^k}$ is the value of Riemann zeta function at $k$, and  $\mbox{ch}_k$ denotes the $k$-th Chern character. %Clearly, $\langle [\mbox{pt}], \hat\Gamma_X\rangle=1$ with respect to the natural pairing  $\langle \cdot,\cdot\rangle : H_*(X)\times H^*(X)\to \mathbb{Q}$   between homology and cohomology.
Gamma conjecture I relates Gromov-Witten invariants (which are deformation invariants but not topological invariants) to the topology of $X$, in the way that the Gamma class of $X$ coincides with the principal asymptotic class of $X$ up to normalization by a scalar; namely
\begin{GammaI}
  Let$X$ be a Fano manifold   satisfying Property $\mathcal{O}$. Then there exists a constant $C\in \mathbb{C}$ such that $\hat\Gamma_X= C\cdot A_X$.
\end{GammaI}
\noindent We refer to \cite{GGI, GaIr} for the various equivalent formulations of the above conjecture.
%{\color{red} It is a bit surprising  that the study of Gamma conjecture I could even be related to the irrationality of the Riemann zeta function $\zeta(k)$ , for which we refer to \cite{Goly, Gal22} for the relation with Apery's famous proof \cite{Apery} in the cases $k=2, 3$.}
As another main result of the present paper, we show the following.
\begin{thm}\label{thmGammconjI}
Gamma conjecture I holds for any del Pezzo surface.
\end{thm}

There have been lots of studies  on  conjecture $\mathcal{O}$ recently. It has been proved for flag varieties $G/P$ of arbitrary Lie type by Cheong and the third named author \cite{ChLi}, and was proved earlier in the special cases of  cominuscule Grassmannians of classical Lie types \cite{Riet, GaGo, Che}.
  Conjecture $\mathcal{O}$ has also been proved recently  for Fano 3-folds of Picard rank one \cite{GoZa}, Fano  complete intersections in projective spaces \cite{GaIr, SaSh, Ke},   horospherical varieties of Picard rank one \cite{LMS, BFSS} and  $3$-dimensional Bott-Samelson varieties \cite{With}. Gamma conjecture I was less studied, and so far has been proved for complex Grassmannians \cite{GGI}, Fano 3-fold of Picard rank one \cite{GoZa}, Fano complete intersections in projective spaces  \cite{GaIr, SaSh, Ke}, and
  toric Fano manifolds that satisfy $B$-model analogue of Property $\mathcal{O}$ \cite{GaIr}.

  The proof of conjecture $\mathcal{O}$ for flag varieties in \cite{ChLi} relied on the well-known Perron-Frobenius  theory  \cite{Perr,Frob} for non-negative real matrices. It requires at least the positivity of  the matrix of the linear operator $\hat c_1$   with respect to some good basis of $H^\bullet(X)$. However, this does not hold even for   the blowup $X_1$ of $\bP^2$ at one point. In the case of del Pezzo surfaces $X_r (1\leq r\leq 8)$, we are able to prove conjecture $\mathcal{O}$ by deriving  a generalization
    of Perron-Frobenius theorem in Theorem \ref{genPFthm} that allows part of the entries of a real matrix to be negative, and providing some vanishing properties of Gromov-Witten invariants in Theorem \ref{vanishingproperty}.
   We remark that the conjecture  could also be directly proved by analysing the characteristic polynomial of $\hat c_1$ as was studied in \cite{BaMa} due to the very well study of   the relevant Gromov-Witten invariants of $X_r$ \cite{GoPa}. However, we expect our method to have further applications for other Fano manifolds, since it will be sufficient to apply our generalized Perron-Frobenius theorem by the study of  a   part  of the Gromov-Witten invariants.
Del Pezzo surfaces are either toric Fano or complete intersections in nice ambient spaces of Picard rank one. This fact enables us to prove Gamma conjecture I by using the mirror techniques proposed by Galkin and Iritani \cite{GaIr}, where the quantum Lefschetz principle \cite{Lee, CoGi} is a key ingredient. Here we would like to point out that the proof for $X_7$ and $X_8$ requires the study of the orbifold version of Gamma conjecture I for certain weighted projective spaces as shown in Theorem \ref{proportionforweightedprojectivespaces}.
A more systemic study of these conjectures (with modification if necessary) for general orbifolds  will be of  independent and great interest.

The present paper is organized as follows. In section 2, we describe the precise statements of conjecture $\mathcal{O}$ and Gamma conjecture I, and review basic facts of del Pezzo surfaces. In section 3, we derive a generalized Perron-Frobenius  theorem, and provide
 vanishing properties of certain Gromov-Witten invariants  of $X_r$. In section 4, we prove conjecture $\mathcal{O}$ by analyzing the corresponding matrix of $\hat c_1$ with respect to a specified basis of $H^*(X)$. Finally in section 5, we prove Gamma conjecture I for del Pezzo surfaces as well as for certain weighted projective spaces.
\section{Preliminaries}
\subsection{Conjecture $\cO$ and Gamma conjecture I for Fano manifolds}
In this subsection, we review the precise statements of   conjecture $\cO$ and   Gamma conjecture I for Fano manifolds,  mainly following \cite{GGI, GaIr}.
\subsubsection{Quantum cohomology} We refer to \cite{CoKa} for more details.

Let $X$ be a Fano manifold, namely a compact complex manifold $X$ whose anticanocial line bundle is ample. Let $\overline{\mathcal{M}}_{0, k}(X, \mathbf{d})$ denote the moduli stack of $k$-pointed genus-zero stable maps $(f: C\to X; p_1, \cdots, p_k)$ of class $\mathbf{d}\in H_2(X,\mathbb{Z})$, which has a coarse moduli space $\overline{M}_{0, k}(X, \mathbf{d})$.
Let $[\overline{M}_{0, k}(X, \mathbf{d})]^{\rm virt}$ be the virtual fundamental class of $\overline{\mathcal{M}}_{0, k}(X, \mathbf{d})$, which is of complex degree $\dim X-3+\int_{\mathbf{d}} c_1(X)+k$ in the Chow group $A_*(\overline{\mathcal{M}}_{0, k}(X, \mathbf{d}))$.
Given classes $\gamma_1, \cdots, \gamma_k\in H^*(X)=H^*(X, \mathbb{Q})$ and nonnegative integers $a_i$ for $1\leq i\leq k$, we have the following  associated gravitational correlator
  $$\langle \tau_{a_1}\gamma_1,\cdots, \tau_{a_k}\gamma_k\rangle_{\mathbf{d}}:=\int_{[\overline{M}_{0, k}(X, \mathbf{d})]^{\rm virt}}\prod_{i=1}^k\big(c_1(\mathcal{L}_i)^{a_i}\cup ev_i^*(\gamma_i)\big).$$
  Here $\mathcal{L}_i$ denotes the line bundle on $\overline{\mathcal{M}}_{0, k}(X, \mathbf{d})$ whose fiber over the stable map  $(f: C\to X; p_1, \cdots, p_k)$ is the cotangent space $T^*_{p_i}C$, and $ev_i$ denotes the $i$-th  evaluation map. When $a_i=0$ for all $i$, the above gravitational correlator becomes an ordinary $k$-pointed Gromov-Witten invariant $\langle  \gamma_1,\cdots,  \gamma_k\rangle_{\mathbf{d}}$ of class $\mathbf{d}$.
The gravitational correlators satisfy a number of axioms and the topological recursion relations (\textbf{TRR}). For the precise statements, we refer to  \cite[section 10.1.2]{CoKa}  for \textbf{Degree Axiom}, \textbf{Divisor Axiom} and \textbf{Fundamental Class Axiom},
 and \cite[Lemma 10.2.2]{CoKa} for \textbf{TRR}, which will be used in the next two sections.

  The (small) quantum cohomogy ring $QH^*(X)=(H^*(X)\otimes_{\mathbb{Q}} \mathbb{Q}[q_1,\cdots, q_m], \star)$ is a deformation of the classical cohomology $H^*(X)$. Here $m=b_2(X)$ is the second Betti number of $X$, and the quantum product of  $\alpha, \beta \in H^*(X)$ is given by\footnote{In terms of the notation in \cite{GGI}, $q_i=e^{h_i}$; the quantum product $\alpha\star_{\tau=\mathbf{0}} \beta$ defined therein coincides with the product $\alpha\star\beta|_{\mathbf{q}=(1, \cdots, 1)}$ here. }
     $$\alpha\star \beta :=\sum_{\mathbf{d}\in H_2(X,\mathbb{Z})}\sum_{i=1}^N \langle \alpha,\beta,\phi_i\rangle_{\mathbf{d}}\phi^i q^{\mathbf{d}}.$$
     Here $\{\phi_i\}_{i=1}^N$ is a homogeneous basis of $H^*(X)$, $\{\phi^i\}$ is its dual basis in $H^*(X)$ with respect to the Poincar\'e pairing, $q^{\mathbf{d}}=\prod_{j=1}^mq_j^{d_j}$ for $\mathbf{d}=(d_1,\cdots, d_m)$ with a basis of effective curve classes of $H_2(X, \mathbb{Z})$ being fixed a prior. We notice that the Gromov-Witten invariant $\langle \alpha,\beta,\phi_i\rangle_{\mathbf{d}}$ vanishes unless $d_i\geq 0$ for all $i$ and   $\deg (\alpha\cup \beta\cup\phi_i)=2(\dim X+\int_{\mathbf{d}} c_1(X))$. In particular, the quantum product is a finite sum and is a polynomial in $\mathbf{q}$. Moreover, the quantum product is independent of choices of the basis $\{\phi_i\}$.

\subsubsection{Conjecture $\cO$} Consider the even part of the cohomology $H^\bullet(X):=H^{\rm even}(X)$ and the finite-dimensional $\mathbb{Q}$-algebra $QH^\bullet(X)=
   (H^\bullet(X), \bullet)$ with the product defined by $\alpha\bullet \beta:=(\alpha\star\beta)|_{\mathbf{q}=\mathbf{1}}$, namely by the evaluation of the quantum product at $\mathbf{1}:=(1,\cdots, 1)$. Let $\hat c_1$ denote the linear operator induced by the first Chern class:
      $$\hat c_1: QH^\bullet(X)\longrightarrow QH^\bullet(X); \, \beta\mapsto c_1(X)\star \beta|_{\mathbf{q}=\mathbf{1}},$$
  which is independent of the choices of bases of effective curve classes.
 \begin{defn}[Property $\cO$]\label{defnrho} For a Fano manifold $X$, we denote by $\rho=\rho(\hat c_1)$ the spectral radius of the linear operator $\hat c_1$, namely
    $$\rho:=\max\{|\lambda|~:~ \lambda\in \mbox{Spec}(\hat c_1)\}\quad\mbox{where}\quad \mbox{Spec}(\hat c_1):=\{\lambda~:~ \lambda\in \bC  \mbox{ is an eigenvalue of } \hat c_1\}.$$
 We say that $X$ satisfies \textbf{Property $\cO$} if the following two conditions are satisfied.
 \begin{enumerate}
   \item $\rho\in \mbox{Spec}(\hat c_1)$ and it is of multiplicity one.
   \item For any $\lambda\in \mbox{Spec}(\hat c_1)$ with $|\lambda|=\rho$, we have $\lambda^s=\rho^s$, where $s$ is the Fano index of $X$, namely $s=\max\{k\in \bZ~:~ {c_1(X)\over k}\in H^2(X,\mathbb{Z})\}$.
 \end{enumerate}
 \end{defn}
 \begin{conjO}[\cite{GGI}]Every Fano manifold satisfies Property $\cO$.
   \end{conjO}
As explained in \cite[Remark 3.1.5]{GGI}, the name ``$\cO$" indicates the structure sheaf $\cO_X$, which is expected to correspond to $\rho$ from the viewpoint of homological mirror symmetry.  Although the above statement concerns about the even part of the cohomology only, it is in fact equivalent to a statement for the whole of $H^*(X)$ \cite{SaSh, GaIr}, and is also equivalent to that for the smaller part $\bigoplus_p H^{p, p}(X)$ of cohomology of Hodge type \cite{GaIr}.

Conjecture $\cO$ for flag variety $G/P$ was proved by Cheong and the third named author \cite{ChLi} by using Perron-Frobenius theorem based on a remark due to Kaoru Ono.

\subsubsection{Gamma conjecture I}
On the trivial $H^\bullet(X)$-bundle over $\mathbb{P}^1$, there is a so-called quantum connection, given by
  $$\nabla_{z\partial_z}=z\partial_z-{1\over z}(c_1(X)\bullet )+\mu.$$
  Here $z$ is an inhomogeneous co-ordinate on $\mathbb{P}^1$ and $\mu$ is the Hodge grading operator defined by $\mu: H^\bullet(X)\to H^\bullet(X); \phi\in H^{2p}(X)\mapsto \mu(\phi)=(p-{\dim X\over 2})\phi$.
  The quantum connection is a meromorphic connection, which is logarithmic at $z=\infty$ and irregular at $z=0$. The space of flat sections can be identified with the cohomology group $H^\bullet(X)$ via the fundamental solution
    $S(z)z^{-\mu}z^{c_1(X)}$ with a unique holomorphic function $S: \mathbb{P}^1\setminus\{0\}\to \mbox{End}(H^\bullet(X))$ (see \cite[Proposition 2.3.1]{GGI} for detailed explanations). The Givental's $J$-function, defined by $J_X(t)=z^{\dim X}\big(S(z)z^{-\mu}z^{c_1(X)}\big)^{-1}1$ where $t={1\over z}$, has the following   expansion around $t=0$.
     \ban
J_X(t)=e^{c_1(X)\log t}\bigg(1+\suml_{i=1}^N\suml_{\mathbf{d}\in H_2(X,\bZ)\setminus\{\mathbf{0}\}}\<\frac{\phi_i}{1-c_1(\mathcal{L}_1)}\>_{0, \mathbf{d}}\phi^it^{\int_{\mathbf{d}}c_1(X)}\bigg).
\nan
%%Strictly speaking,  we have been taking   the restriction of the original Givental's big $J$-function to the anti-canonical line $c_1(X)\log t$ and taking the evaluation  at $\mathbf{q}=\mathbf{1}$ here.
\begin{remark}
  Strictly speaking,   Givental's big $J$-function is a formal function of $\tau\in H^\bullet(X)$ taking values in $H^\bullet(X)$,   given by
  \ban
J_X(\tau,\hbar;Q)=\hbar+\tau+\suml_{i=1}^N\phi^i\suml_{\mathbf{d}\in H_2(X,\bZ)}Q^\mathbf d\suml_{m=0}^\infty{1\over m!}\<\frac{\phi_i}{\hbar-c_1(\mathcal{L}_1)},\underbrace{\tau,\cdots,\tau}_m\>_{\mathbf d},
\nan
It can be reduced to the form  $J_X(\tau,\hbar;Q)=\hbar e^{\frac{\tau}{\hbar}}\Big(\mathbbm1+\suml_{i=1}^N\phi^i\suml_{\mathbf{d}\in H_2(X,\bZ)}\hspace{-0.2cm}Q^\mathbf de^{\int_\mathbf d\tau}\<\frac{\phi_i}{\hbar(\hbar-c_1(\mathcal{L}_1)}\>_\mathbf d\Big)$ for   $\tau\in H^0(X)\oplus H^2(X)$, after applying the Fundamental Class Axiom and the Divisor Axiom.
The Givental's $J$-function in this paper is obtained from the original one by setting
\ban
J_X(t)=J_X(\tau=c_1(X)\log t,\hbar=1;Q=\mathbf{1}).
\nan
\end{remark}
\iffalse More precisely, recall that Givental's big $J$-function is formal function of $\tau\in H^\bullet(X)$ taking values in $H^\bullet(X)$, which is given by ?\ban
J_X(\tau,\hbar;Q)=\hbar+\tau+\suml_{i=1}^N\phi^i\suml_{\mathbf{d}\in H_2(X,\bZ)}Q^\mathbf d\suml_{m=0}^\infty\<\frac{\phi_i}{\hbar-\psi_1},\underbrace{\tau,\cdots,\tau}_m\>_{\mathbf d},
\nan
where $\psi_1=c_1(\mathcal{L}_1)$.
In particular, for $\tau\in H^0(X)\oplus H^2(X)$, we can use the fundamental axiom and the divisor axiom to obtain
\ban
J_X(\tau,\hbar;Q)=\hbar e^{\frac{\tau}{\hbar}}\bigg(\mathbbm1+\suml_{i=1}^N\phi^i\suml_{\mathbf{d}\in H_2(X,\bZ)}Q^\mathbf de^{\int_\mathbf d\tau}\<\frac{\phi_i}{\hbar(\hbar-\psi_1)}\>_\mathbf d\bigg).
\nan
So
\ban
J_X(t)=J_X(\tau=c_1(X)\log t,\hbar=1;Q=1).
\nan
\fi

\begin{prop}[Proposition 3.8 of \cite{GaIr}]\label{propJexpansion}
  For any Fano manifold $X$ satisfying Property $\cO$,  Givental's $J$-function $J_X(t)$ has an asymptotic expansion of the form
   \begin{equation}\label{expandJ}
      J_X(t)=Ct^{-\dim X\over 2}e^{\rho t}(A_X+\alpha_1 t^{-1}+\alpha_2t^{-2}+\cdots)
   \end{equation}
  as $t\to +\infty$ on the positive real line, where $C$ is a non-zero constant and $\alpha_i\in H^\bullet(X)$. % and $\rho=\rho(\hat c_1)$ is the spectral radius defined in Definition \ref{defnrho}.
\end{prop}

The Gamma class \cite{Libg, Lu, Iri}  is a real characteristic class defined for an almost complex manifold. It is defined by Chern roots $x_1,\cdots, x_n$ of the tangent bundle $TX$ of $X$ and
  Euler's $\Gamma$-function  $\Gamma(z)=\int_0^\infty e^tt^{z-1}dt$, and has the following expansion:
  $$\hat \Gamma_X:=\prod_{i=1}^n\Gamma(1+x_i)=\exp\big(-C_{\rm eu}c_1(X)+\sum_{k=2}^\infty (-1)^k(k-1)!\zeta(k)\mbox{ch}_k(TX)\big)  \in H^\bullet (X, \mathbb{R})$$
where $C_{\rm eu}$ is the  Euler-Mascheroni  constant, $\zeta(k)=\sum_{n=1}^\infty{1\over n^k}$ is the value of Riemann zeta function at $k$, and  $\mbox{ch}_k$ denotes the $k$-th Chern character. %Clearly, $\langle [\mbox{pt}], \hat\Gamma_X\rangle=1$ with respect to the natural pairing  $\langle \cdot,\cdot\rangle : H_*(X)\times H^*(X)\to \mathbb{Q}$   between homology and cohomology.
There are various equivalent ways to describe  Gamma conjecture I. Here we introduce the one given in \cite[Corollary 3.6.9 (3)]{GGI}. For $\alpha, \beta\in H^*(X)$, we say $\alpha \varpropto \beta$ if $\alpha= b\cdot \beta$ for some $b\in \mathbb{C}$.
\begin{GammaI}
Let $X$ be a Fano manifold satisfying Property $\cO$. Then
\ban
\hat\Gamma_X\varpropto\lim\limits_{t\rightarrow+\infty}t^{\frac{\dim X}2}e^{-\rho t}J_X(t).
\nan
   %the class $A_X$ in the asymptotic expansion \eqref{expandJ} satisfies
    %$$\langle [\mbox{pt}], A_X\rangle\neq 0\quad\mbox{and}\quad\hat \Gamma_X={A_X\over \langle [\mbox{pt}], A_X\rangle}$$
\end{GammaI}

\if{
\subsubsection{Orbifold version of conjecture $\cO$ and Gamma conjecture I}

A Fano orbifold is a connected, compact, complex orbifold with positive anitcanonical bundle.

It was known that some Fano orbifolds with nontrivial orbifold fundamental groups do not satisfy part (1) of Property $\cO$ \cite{Gal,GGI}, and Galkin-Golyshev-Iritani made the following conjecture \cite[Remark 3.1.8]{GGI}.

\begin{conjecture}\label{orbifoldconjectureO}
For a Fano orbifold $\mathcal X$, $\rho\in \mbox{Spec}(\hat c_1)$ with multiplicity equal to the number of conjugacy classes of $\pi_1^{orb}(\mathcal X)$.
\end{conjecture}
\begin{remark}
A Fano manifold is always simply connected. However, it is unknown that whether or not the orbifold fundamental group of a Fano orbifold is finite. We thank Zhiyu Tian for informing us this open question.
\end{remark}

The second named author has proved that Fano orbi-curves do not satisfy part (2) of Property $\cO$, and we make the following conjecture.

\begin{conjecture}
For a Fano orbifold $\mathcal X$, the rotational symmetry group of $\mbox{Spec}(\hat c_1)$ is a cyclic group with order equal to the Fano index of $\mathcal X$.
\end{conjecture}

We refer readers to \cite[(23)]{Iri} for the explicit definition of the Gamma class $\hat\Gamma_\mathcal X$ of an almost complex orbifold $\mathcal X$.
\begin{conjecture}
Let $\mathcal X$ be a Fano orbifold with $\pi_1^{orb}(\mathcal X)=\{1\}$. If $\mathcal X$ satisfies Conjecture $\cO$, then $$\langle [\mbox{pt}], A_\mathcal X\rangle\neq 0\quad\mbox{and}\quad\hat \Gamma_\mathcal X={A_\mathcal X\over \langle [\mbox{pt}], A_\mathcal X\rangle}$$

\end{conjecture}
}\fi

\subsection{Del Pezzo surfaces}
It is well-known that any one-dimensional Fano manifold is isomorphic to the complex projective line $\mathbb{P}^1$.   A Fano manifold of dimension $2$ is called a \textit{del Pezzo surface}. It is either isomorphic to $\mathbb{P}^1\times \mathbb{P}^1$, or the blowup $X_r$ of $\mathbb{P}^2$ at $r$ points in general position ($0\leq r\leq 8$). We will exclude  $\mathbb{P}^1\times \mathbb{P}^1$ and $\mathbb{P}^2$ in this subsection.
\subsubsection{Basic topology} Curves in a surface are  divisors.  For convenience,  we will use the same notation for anyone of a divisor of $X_r$, its divisor class in $H^2(X_r, \mathbb{Z})$ and its curve class in $H_2(X_r, \mathbb{Z})$,  whenever there is no confusion. For instance in $D\cdot D'=\int_XD\cup D'$, we can easily read off the left (resp. right) hand side as the intersection product (resp. Poincar\'e pairing) of the divisors (resp. divisor  classes) $D, D'$.

Let  $H$ denote the pullback to $X_r$ of the  hyperplane class of $\bP^2$, and let $E_1,\cdots, E_r$ be the exceptional divisors.  Together with the
Poincar\'e dual $\mathbbm{1}:=[X_r]\in H^0(X_r,\mathbb{Z})$ and  $[{pt}]\in H^4(X_r, \mathbb{Z})$   of the corresponding   homology classes, they form a $\mathbb{Z}$-basis:
 $$H^*(X_r,\mathbb{Z})=H^{\rm even}(X_r, \mathbb{Z})=\mathbb{Z}\mathbbm{1}\oplus \mathbb{Z}[pt]\oplus \mathbb{Z}H\oplus \mathbb{Z}E_1\oplus\cdots\oplus \mathbb{Z}E_r $$
 The first Chern class of $X_r$ is given by $c_1:=c_1(X_r)=3H-\sum_{i=1}^rE_i$.
  We have $$H\cdot E_i=E_i\cdot E_j=0,\quad H\cdot H=-E_i\cdot E_i=1,\quad \mbox{for all }i, j\in\{1,\cdots, r\} \mbox{ with }i\neq j.$$
It follows that $X_r$ is of degree $c_1\cdot c_1=9-r$,  and   its Fano index    equals $1$  as  $c_1\cdot E_1=1$.

\subsubsection{Geometric interpretations}\label{geometricinterpretation} There are various geometric descriptions for del Pezzo surfaces $X_r$, where $1\leq r\leq 8$. For instance, all of them can be realized as complete intersections in   nice spaces as follows.

\begin{itemize}
  \item [$X_1$:]    degree $(1,1)$ hypersurface in $\bP^1\times \bP^2$;
  \item [$X_2$:]   complete intersection of   divisors of  degree $(1, 0, 1)$ and $(0, 1, 1)$ in  $\bP^1\times \bP^1\times \bP^2$;
   \item [$X_3$:]   complete intersection of  two divisors of  degree $(1,  1)$   in  $\bP^2\times \bP^2$;

   \item [$X_4$:] complete intersection of four hyperplanes in  Grassmannian $Gr(2,5)\subset\bP^9$ (embedded by Pl\"ucker);

 \item [$X_5$:] complete intersection of two quadrics in $\bP^4$;
 \item [$X_6$:]   cubic surface in $\bP^3$;

 \item [$X_7$:] hypersurface of degree $4$ in the weighted projective space $\bP(1,1,1,2)$;

 \item [$X_8$:] hypersurface of degree $6$ in the weighted projective space  $\bP(1,1,2,3)$.

\end{itemize}
  We refer to \cite[Chapter 3.2]{IsPr} and \cite{Cord} for the above descriptions of  $X_r$ with $4\leq r\leq 8$. Toric del Pezzo surfaces are precisely those $X_r$ with $1\leq r\leq 3$, and the descriptions are also known to the experts. Here we provide a proof for $X_2$, which is relatively less studied. The method    works for $X_1$ and $X_3$ as well.

\begin{proof}[Proof of $X_2$ being a complete intersection]

The  complete intersection $Z$ of  two generic   divisors of  degree $(1, 0, 1)$ and $(0, 1, 1)$ in  $Y:=\bP^1\times \bP^1\times \bP^2$ is a smooth projective variety.
By the adjunction formula,  we obtain the  anti-canonical divisor $-K_Z=\mathcal{O}_{{Y}}(H_1+H_2+H_3)|_Z$, where $H_i$   is the natural pull-back of the hyperplane class of the $i$-th factor  to $Y$.
   Clearly, $-K_Z$ is ample, and hence $Z$ is a del Pezzo surface.

  The degree $(-K_Z)^2$ of $Z$ can be computed by the intersection product of divisors in $Y$ as follows.  Notice that $H_1^2=H_2^2=H_3^3=0$, $H_1H_2H_3^2=1$ and $H_iH_j=H_jH_i$, we have
   \begin{align*}
      (-K_Z)^2&=(H_1+H_2+H_3)^2(H_1+H_3)(H_2+H_3) \\
      &=(H_1^2+H_2^2+H_3^2+2H_1H_2+2H_1H_3+2H_2H_3)(H_1H_2+H_1H_3+H_2H_3+H_3^2)\\
      &=7H_1H_2H_3^2=7
   \end{align*}
   It follows that $Z$ is a del Pezzo surface of degree $7$.
\end{proof}
\subsubsection{Quantum cohomology}\label{QHdelPezzo}

There has been well  study of the quantum cohomology ring $QH^*(X_r)$ of $X_r$ in \cite{CrMi,GoPa}. Therein we can immediately read off or easily deduce the multiplication table of the quantum product for small $r$. We notice that $\mathbbm{1}$ is the identity element in $QH^*(X_r)$ for any $X_r$.

\begin{example}\label{matX1}
  A basis of  $H^*(X_1)$ is given by $\{\mathbbm{1}, H, E_1, [pt]\}$. In $QH^*(X_1)$,   we have
\ban
H\star H=[pt]+q^{H-E_1},\quad H\star E_1=q^{H-E_1},\quad  H\star[pt]=(H-E_1)q^{H-E_1}+q^H,\\
E_1\star E_1=-[pt]+E_1q^{E_1}+q^{H-E_1},\qquad\quad\quad  E_1\star [pt]=(H-E_1)q^{H-E_1}.
\nan
For conjecture $\cO$, we  concern about the operator $\hat c_1$ on $QH^\bullet(X_1)$ induced by the quantum multiplication by  $c_1=3H-E_1$ with evaluation of all quantum variables at $1$. We have
\ban
\hat c_1[\mathbbm1,H,E_1,[pt]]=[\mathbbm1,H,E_1,[pt]]\left[\begin{array}{cccc}0&2&2&3\\3&0&0&2\\-1&0&-1&-2\\0&3&1&0\end{array}\right].
\nan

\noindent We denote by $\tilde M_1$ the matrix above, while denote by $M_1$   the matrix with respect to another $\mathbb{Z}$-basis $[\mathbbm1, H-E_1,E_1,[pt]]$. It follws that
\ban
M=\left[\begin{array}{cccc}1&0&0&0\\0&1&0&0\\0&-1&1&0\\0&0&0&1\end{array}\right]^{-1}\cdot\tilde M\cdot \left[\begin{array}{cccc}1&0&0&0\\0&1&0&0\\0&-1&1&0\\0&0&0&1\end{array}\right]=\left[\begin{array}{cccc}0&0&2&3\\3&0&0&2\\2&1&-1&0\\0&2&1&0\end{array}\right].
%%\Rightarrow M^3=\left[\begin{array}{cccc}26&1&7&28\\28&26&1&8\\7&21&5&1\\1&7&21&26\end{array}\right].
\nan

\end{example}

\begin{example}\label{matX2} A basis of  $H^*(X_2)$ is given by $\{\mathbbm{1}, H, E_1, E_2, [pt]\}$. In $QH^*(X_2)$,   we have
\ban
H\star H&=&[pt]+(H-E_1-E_2)q^{H-E_1-E_2}+q^{H-E_1}+q^{H-E_2},\\
H\star E_1&=&(H-E_1-E_2)q^{H-E_1-E_2}+q^{H-E_1},\\
H\star[pt]&=&(H-E_1)q^{H-E_1}+(H-E_2)q^{H-E_2}+q^H,\\
E_1\star E_1&=&-[pt]+E_1q^{E_1}+(H-E_1-E_2)q^{H-E_1-E_2}+q^{H-E_1},\\
E_1\star E_2&=&(H-E_1-E_2)q^{H-E_1-E_2},\\
E_1\star[pt]&=&(H-E_1)q^{H-E_1}.
\nan
 Let $\tilde M_2$ and  $M_2$   denote  the matrices  with respect to the corresponding $\mathbb{Z}$-bases:
 \begin{align*}
    \hat c_1 [\mathbbm1,H,E_1,E_2,[pt]]&=[\mathbbm1,H,E_1,E_2,[pt]]\tilde M_2,\\
   \hat c_1 [\mathbbm1,2H-E_1-E_2,E_1,E_2,[pt]]&=[\mathbbm1,2H-E_1-E_2,E_1,E_2,[pt]] M_2.
 \end{align*}
Then $\tilde M_2$ is directly read off from the above  table (by setting all quantum variables as $1$), and $M_2$ is obtained after a simple base change from $\tilde M_2$. They are precisely given by
\ban
\tilde M_2=\left[\begin{array}{ccccc}0&4&2&2&3\\3&1&1&1&4\\-1&-1&-2&-1&-2\\-1&-1&-1&-2&-2\\0&3&1&1&0\end{array}\right],\qquad
M_2=\left[\begin{array}{ccccc}0&4&2&2&3\\{3\over 2}&0&{1\over 2}&{1\over 2}&2\\{1\over 2}&1&-{3\over 2}&-{1\over 2}&0\\{1\over 2}&1&-{1\over 2}&-{3\over 2}&0\\0&4&1&1&0\end{array}\right].
%%M_2^3=\left[\begin{array}{ccccc}25&3&30&30&56\\53&18&10&10&63\\30&23&2&3&33\\30&23&3&2&33\\3&7&23&23&25\end{array}\right].
\nan

\end{example}

\begin{example}\label{matX3}
 The quantum multiplication table of  $QH^*(X_3)$ with respect to the  basis   $\{\mathbbm{1}, H, E_1, E_2, E_3, [pt]\}$ can be read off from the following together with a permutation symmetry among the exceptional divisor classes $E_i$.
  \ban
H\star H&=&[pt]+\suml_{1\leq i\leq3}(H-\suml_{j\neq i}E_j)q^{H-\suml_{j\neq i}E_j}+\suml_{1\leq i\leq3}q^{H-E_i},\\
H\star E_1&=&(H-E_1-E_2)q^{H-E_1-E_2}+(H-E_1-E_3)q^{H-E_1-E_3}+q^{H-E_1},\\
H\star [pt]&=&\suml_{1\leq i\leq3}(H-E_i)q^{H-E_i}+q^H+2q^{2H-E_1-E_2-E_3},\\
E_1\star E_1&=&-[pt]+E_1q^{E_1}+(H-E_1-E_2)q^{H-E_1-E_2}+(H-E_1-E_3)q^{H-E_1-E_3}+q^{H-E_1},\\
E_1\star E_2&=&(H-E_1-E_2)q^{H-E_1-E_2},\\
E_1\star [pt]&=&(H-E_1)q^{H-E_1}+q^{2H-E_1-E_2-E_3}.
\nan
\end{example}
We again denote by $\tilde M_3$ the matrix  of $\hat c_1$ with respect to the above basis, while denote by $M_3$ the matrix with respect to the following $\mathbb{Q}$-basis, in contrast to that for $X_1, X_2$.
\begin{align*}
%    \hat c_1 [\mathbbm1,H-E_1-E_2,E_1,E_2,[pt]]&=[\mathbbm1,H,E_1,E_2,[pt]]\tilde M_2,\\
   \hat c_1 [\mathbbm1, c_1, E_1, E_2, E_3, [pt]]&=[\mathbbm1, c_1, E_1, E_2, E_3, [pt]] M_3.
 \end{align*}

Since $c_1=3H-E_1-E_2-E_3$, the matrices $\tilde M_3$ and $M_3$ are respectively given by
\ban
\tilde M_3=\left[\begin{array}{cccccc}
0&6&2&2&2&6\\
3&3&2&2&2&6\\
-1&-2&-3&-1&-1&-2\\
-1&-2&-1&-3&-1&-2\\
-1&-2&-1&-1&-3&-2\\
0&3&1&1&1&0
\end{array}\right],\, M_3=\left[\begin{array}{cccccc}
0&12&2&2&2&6\\
1&1&2/3&2/3&2/3&2\\
0&0&-7/3&-1/3&-1/3&0\\
0&0&-1/3&-7/3&-1/3&0\\
0&0&-1/3&-1/3&-7/3&0\\
0&6&1&1&1&0
\end{array}\right].
%\Rightarrow M_3^2=\left[\begin{array}{cccccc} 12&48&8&8&8&24\\1&25&8/3&8/3&8/3&8\\0&0&17/3&5/3&5/3&0\\0&0&5/3&17/3&5/3&0\\0&0&5/3&5/3&17/3&0\\6&6&1&1&1&12\end{array}\right].
\nan

\section{Main technical results}
\subsection{A generalization of Perron-Frobenius theorem} By a matrix in this section, we always mean a square real finite matrix. The spectral radius of a matrix $M$ is given by
$$\rho(M):=\max\{|\lambda|~|~ \lambda \mbox{ is an eigenvalue of }M\}.$$
The  matrix $M$ is called \textit{reducible}, if there exists a permutation matrix $P$ such that $P^tMP$ is of the  form
$\left(\begin{array}{cc}
A&B\\
0&D
\end{array}\right)$   where $A,D$ are square submatrices.
 The matrix $M$ is called  \textit{irreducible}  if it is not reducible.
 The well-known Perron-Frobenius theory  \cite{Perr,Frob} concerns about properties on eigenvalues and eigenvectors of nonnegative irreducible matrices including the following proposition. It plays an essential role in the proof of conjecture $\mathcal{O}$ for flag varieties in \cite{ChLi}.

\begin{prop}[Theorem 1.5   of \cite{Sene}]\label{Perron-Frob} Let $M$ be an irreducible  nonnegative matrix. Then the spectral $\rho(M)$ itself is an eigenvalue of $M$ with multiplicity one.
\end{prop}

There have been various  extensions of Perron-Frobenius theory (see \cite{ElSz} and references therein). Here we  give one more extension, which is new to our knowledge.

\begin{thm}[Generalized Perron-Frobenius Theorem]\label{genPFthm}
  Let $T=(t_{ij})$ be an $n\times n$ real matrix that satisfies the following:
  \begin{enumerate}
    \item $\sum_{i=1}^nt_{ij}>0$ for $j=1,\cdots, n$;

    \item $T^k$ is an irreducible nonnegative matrix for some positive integer $k$.
  \end{enumerate}
  Then   the spectral radius $\rho(T)$ itself is an eigenvalue of $T$ with multiplicity one.
\end{thm}

\begin{remark}\label{rmkforPF}
  The condition (1)  can be replaced by $\sum_{j=1}^nt_{ij}>0$ for $i=1,\cdots, n$. Furthermore if the integer $k$ in condition (2) is odd, then condition (1) is abundant, and the statement is an easy consequence of the Perron-Frobenius Theorem.

   Condition (1) was discovered by  the fourth named author, who also showed the existence of a positive eigenvector corresponding to $\rho(T)$ in his  Bachelor Thesis \cite{Yang}.
\end{remark}

For column vectors $\mathbf{x}=(x_i)_{n\times 1}, \mathbf{y}=(y_i)_{n\times 1}$ in $\mathbb{R}^n$, we say $\mathbf{x}\leq \mathbf{y}$  (resp. $\mathbf{x}<\mathbf{y}$)
if $x_i\leq y_i$ (resp. $x_i<y_i$) for all $i\in\{1, \cdots, n\}$.

\begin{lemma}\label{subinvlem}
  Let $T$ be an $n\times n$  irreducible nonnegative matrix.
  \begin{enumerate}
    \item There exists $m\in \mathbb{Z}_{>0}$ such that $(I_n+T)^m$ is a matrix with all entries positive.
    \item Let $s\in \mathbb{R}_{>0}$,  and $\mathbf{y}\geq \mathbf{0}$ be a nonzero vector satisfying
          $T\mathbf{y}\leq  s \mathbf{y}$. Then \\
          {\upshape (a)} $\mathbf{y} > \mathbf{0}$; {\upshape (b)} $s\geq  \rho(T)$. Moveover, $s=\rho(T)$ if and only if $T\mathbf{y}=s\mathbf{y}$.
  \end{enumerate}
\end{lemma}
\noindent Here the first statement is a well-known property for irreducible nonnegative matrices (see e.g.   \cite[Theorem 1.4]{Sene}); the second statement   is referred to as ``The Sub-invariance Theorem" (see e.g.  \cite[Theorem 1.6]{Sene}), which asserts that any irreducible nonnegative matrix has a unique nonnegative eigenvector up to a positive scalar. %To prove it, we will need the following lemma.
\bigskip

\begin{proof}[Proof of Theorem \ref{genPFthm}]
We first show that $T$ has a positive eigenvalue as follows. The argument  is similar to that for Theorem 1.1 (a) of   \cite{Sene}.

 Consider the upper-semicontinuous function defined by
     $$r: \mathbb{R}^n_{\geq 0}\setminus \{\mathbf{0}\}\to \mathbb{R};\quad \mathbf{x}=(x_1, \cdots, x_n)^t\mapsto r(\mathbf{x}):=\min_{1\leq j\leq n\atop x_j\neq 0}{\sum_i x_i t_{ij}\over x_j}.$$

    By assumption (1) and the definition of $r(\mathbf{x})$, $r(\mathbf{1})>0$ where $\mathbf{1}:=(1,\cdots, 1)^t$. Denote $K:=\max\limits_{1\leq i\leq n} \sum_j |t_{ij}|$. Since $x_j r(\mathbf{x})\leq \max\{0, \sum_{i} x_it_{ij}\} \leq \sum_{i}x_i|t_{ij}|$ for all $j$, we have
     $r(\mathbf{x})\sum_{j}x_j\leq \sum_j\sum_ix_i|t_{ij}|\leq \sum_{i}x_iK$ and hence
      $r(\mathbf{x})\leq K$. Therefore for
      $$\rho:= \sup_{\mathbf{x}\in \mathbb{R}^n_{\geq 0}\setminus \{\mathbf{0}\}}r(\mathbf{x})=\sup_{\mathbf{x}\geq \mathbf{0}\atop \mathbf{x}^t\mathbf{x}=1}r(\mathbf{x}),$$
     we have  $0<r(\mathbf{1})\leq \rho\leq K<+\infty.$
    Since $r(\mathbf{x})$ is upper-semicontinuous on the compact subset $\{\mathbf{x}~|~\mathbf{x}\geq 0, \mathbf{x}^t\mathbf{x}=1\}$,  the supremum $\rho$  is actually attained for some $\hat{\mathbf{x}}$ in this compact region. Consequently, we have
    $\sum_{i}\hat x_i t_{ij}\geq \hat x_j r(\hat{\mathbf{x}})=\rho \hat x_j$ for all $j$, namely $\mathbf{z}^t:=\hat{\mathbf{x}}^t T -\rho \hat{\mathbf{x}}^t\geq \mathbf{0}^t$.

  Assume
     $\mathbf{z}\neq \mathbf{0}.$ By Lemma \ref{subinvlem}, all entries of $(I_n+T^k)^m$ are positive   for some $m$. Thus
      $$\mathbf{0}^t<\mathbf{z}^t(I_n+T^k)^m=\hat{\mathbf{x}}^t (I_n+T^k)^mT -\rho \hat{\mathbf{x}}^t(I_n+T^k)^m \quad\mbox{and}\quad
      \tilde{\mathbf{x}}:=\hat{\mathbf{x}}^t(I_n+T^k)^m>\mathbf{0},$$
ahence
     ${\sum_{i}\tilde x_it_{ij}\over \tilde x_j}>\rho$ for all $1\leq j\leq n$. It follows that   $r(\hat{\mathbf{x}})>\rho$, which  is a contradiction to the definition of $\rho$. Hence, we have $\mathbf{z}=\mathbf{0}$, implying that $\rho$ is an eigenvalue of $T$.

     Now we have $\hat{\mathbf{x}}^t T=\rho \hat{\mathbf{x}}^t$. It follows that $(T^k)^{\rm{t}} \hat{\mathbf{x}}=\rho^k \hat{\mathbf{x}}$, where we notice $\rho^k>0$ and $\hat{\mathbf{x}}\geq \mathbf{0}$. Therefore by assumption (2) and Lemma \ref{subinvlem} (2),
     we have $\rho^k=\rho((T^k)^{\rm{t}})=\rho(T^k)$. For any eigenvalue $\lambda$ of $T$, $\lambda^k$ is an eigenvalue of $T^k$. It follows that
       $|\lambda^k|\leq \rho(T^k)=\rho^k$, and hence $|\lambda|\leq \rho$. That is, $\rho(T)=\rho$. Moreover, $\rho^k$ is an eigenvalue of $T^k$ of    multiplicity one by Proposition \ref{Perron-Frob}. Thus  $\rho$ is an eigenvalue of $T$ of multiplicity one.
\end{proof}
The next property is a well-known consequence of the Perron-Frobenius theory on nonnegative matrices. The result was due to Perron when all the entries  $m_{ij}$ are  positive.
\begin{prop}\label{propuniquerho}
  Let $M=\big(m_{ij}\big)$ be an irreducible nonnegative matrix. Suppose $M$ has exactly $h$ eigenvalues of modulus $\rho(M)$ with multiplicities counted. If there exists $i$ such that $m_{ij}>0$ for any $j$, then   $h=1$, namely  $M$ has a unique   eigenvalue  with modulus $\rho(M)$ given by the simple eigenvalue $\rho(M)$ itself.
\end{prop}
\begin{proof} The number $h$ is traditionally called the {\it index of imprimitivity} (or {\it period}) of the irreducible matrix $M$.
  If  $h>1$, then  there exists a permutation matrix $P$ such that $PMP^t$ is of the following form   (see e.g.    Theorems 3.1 of Chapter 3 of \cite{Minc}):
$$
 \left[\begin{array}{cccccc}
0&M_{12}&0&\cdots&0&0\\
0&0&M_{2,3}&\cdots&0&0\\
\vdots&&&\ddots&0&\vdots\\
0&0&&\cdots&0&M_{h-1,h}\\
M_{h1}&0&&\cdots&&0
\end{array}\right].
$$
However, the hypothesis implies that $PMP^t$ always contains a row of positive numbers, which makes a contradiction.
\end{proof}

\begin{example}
  The matrix $M_1$ (resp. $M_2, M_3$)  in Example \ref{matX1} (resp. \ref{matX2}, \ref{matX3})   obviously satisfies the condition (1) in Proposition \ref{genPFthm}. By direct calculations, we have
  $$M_1^3=\left[\begin{array}{cccc}26&1&7&28\\28&26&1&8\\7&21&5&1\\1&7&21&26\end{array}\right],\quad M_2^2=\left[\begin{array}{ccccc}8&16&1&1&8\\{1\over 2}&15&4&4&{9\over 2}\\{1\over 2}&0&4&3&{7\over 2}\\{1\over 2}&0&3&4&{7\over 2}\\7&2&0&0&8\end{array}\right],$$
 $$\mbox{and}\hspace{2.5cm} M_3^2=\left[\begin{array}{cccccc} 12&48&8&8&8&24\\1&25&8/3&8/3&8/3&8\\0&0&17/3&5/3&5/3&0\\0&0&5/3&17/3&5/3&0\\0&0&5/3&5/3&17/3&0\\6&6&1&1&1&12\end{array}\right].$$
 Both $M_1^3$ and $M_2^2$
   are irreducible nonnegative matrices, whose first row   consist of  positive numbers. Thus for $i\in\{1,2\}$,   $\rho(M_i)$ is an eigenvalue of $M_i$ of multiplicity one  by Theorem \ref{genPFthm}, and it is the unique eigenvalue of $M_i$ with modulus $\rho(M_i)$ by Proposition \ref{propuniquerho}. Although the matrix $M_3^2$ is reducible, we can still conclude the same property for $\rho(M_3)$, by applying a consequence of Theorem \ref{genPFthm}  (see Proposition \ref{propMatrixdelPezzo}).
\end{example}

\begin{remark}
  The matrix $\tilde M_1$   in Example   \ref{matX1} or another matrix $\tilde M_1'$ with respect to the $\mathbb{Q}$-basis  $\{\mathbbm1,  c_1, E_1, [pt]\}$ of
   $H^*(X_1)$ both satisfy the condition (1) in Theorem \ref{genPFthm}. However,   there does not exist positive integer $k$ such that $\tilde M_1^k$ or $(\tilde M_1')^k$ is a nonnegative matrix. The situation for $X_2$ is the same. The condition  (1) in Theorem \ref{genPFthm} even fails for $\tilde M_3$.

\end{remark}

\subsection{Vanishing properties of Gromov-Witten invariants of del Pezzo surfaces}
In this subsection, we assume  $3\leq r\leq 8$ and simply denote  $c_1:=c_1(X_r)$.
We denote
$$\mbox{Amb}_r:=\mathbb{Q}\mathbbm1\oplus \mathbb{Q}c_1\oplus \mathbb{Q}[pt]\quad\mbox{and}\quad \mbox{Prim}_r:=\{\gamma\in H^2(X_r):c_1\cup\gamma=0\}.$$
Then we obtain $
H^*(X_r)=\mbox{Amb}_r\oplus\mbox{Prim}_r
$ as an orthogonal decomposition  of vector spaces  with respect to the Poincar\'e pairing. Moreover,  $\mbox{Amb}_r$
is the subalgebra of $H^*(X_r)$ generated by $c_1$.

The main result of this subsection is the following, which will be used to simplify part of the entries of the matrix corresponding to the operator $\hat c_1$ in section \ref{propMatrixdelPezzo},
and to   guarantee   the vanishing property of the primitive part of the $J$-function of $X_r$  in section \ref{GammaconjectureIfornontoricdelPezzosurfaces}.

\begin{thm}\label{vanishingproperty}
For any {\upshape $m, a, a_i\in \mathbb{Z}_{\geq 0},  \gamma\in \mbox{Prim}_r$}  and   {\upshape $\gamma_i\in\mbox{Amb}_r$}, $i=1,\cdots, m$, we have

%%Let $m\geq0$ and $a,a_1,\cdots,a_m\geq0$. For any $\gamma\in\mbox{Prim}_r$  and any $\gamma_1,\cdots,\gamma_m\in\mbox{Amb}_r$, we have
\ba\label{vanishingpropertyformula}
\suml_{A\in H_2(X_r, \bZ)}\<\tau_a(\gamma)\prodl_{i=1}^{m}\tau_{a_i}(\gamma_i)\>_{A}=0.
\na
\end{thm}
\begin{remark} There are only finitely many nonzero terms in the above summation by \cite[Corollary 1.19]{KoMo} together with the observation that
%Note that only effective curve classes can support non-zero Gromov-Witten invariants from the definition. Together with the Degree Axiom, this implies that on LHS of \eqref{vanishingpropertyformula},
   $\<\tau_a(\gamma)\prodl_{i=1}^{m}\tau_{a_i}(\gamma_i)\>_{A}$ is nonzero  only if $A$ is an effective class with
$
\int_Ac_1=(a+\frac12\deg\gamma)+\suml_{i=1}^m(a_i+\frac12\deg\gamma_i)-m.
$
 %So it follows from \cite[Corollary 1.19]{KoMo} that only finitely many terms on LHS of \eqref{vanishingpropertyformula} are non-zero.

\end{remark}

The rest of this subsection is devoted to a proof of Theorem \ref{vanishingproperty}.

\begin{lemma}\label{degreezero}
For any {\upshape $m, a, a_i\in \mathbb{Z}_{\geq 0},  \gamma\in \mbox{Prim}_r$}  and   {\upshape $\gamma_i\in\mbox{Amb}_r$}, $i=1,\cdots, m$, we have
%Let $m\geq0$ and $a,a_1,\cdots,a_m\geq0$. Then for any $\gamma\in\mbox{Prim}_r$, and $\gamma_1,\cdots,\gamma_m\in\mbox{Amb}_r$, we have
\ban
\<\tau_a(\gamma)\prodl_{i=1}^{m}\tau_{a_{i}}(\gamma_i)\>_{0}=0
\nan
\end{lemma}
\begin{proof}
Note that the obstruction theory on $\overline M_{0,m+1}(X_r,0)=X_r\times\overline M_{0,m+1}$ is trivial. So the vanishing result follows from the definition of $\mbox{Prim}_r$.
\end{proof}

%Before we consider more general cases, let us point out
Gromov-Witten invariants of $X_r$ admit two types of symmetry. %   Recall that for $A\in H_2(X_r,\bZ)$, we can write uniquely $A=dH-\suml_{k=1}^rd_kE_k$.
The classical Cremona transformation of $\bP^2$ induces an involution $\sigma$ on $H_2(X_r,\bZ)$,  given by
$$\sigma(dH-\suml_{k=1}^rd_kE_k)
=%(2d-d_1-d_2-d_3)H-(d-d_2-d_3)E_1-(d-d_1-d_3)E_2-(d-d_1-d_2)E_3-\suml_{4\leq k\leq r}d_kE_k.
(2d-d_1-d_2-d_3)H-\sum_{1\leq k\leq 3}(d-d_1-d_2-d_3+d_k)E_k-\suml_{4\leq k\leq r}d_kE_k.
$$
This leads to the first type of symmetry below,  by  the argument in \cite[Section 5.1]{GoPa}.
\begin{lemma}\label{symmetry1}
For any {\upshape $m, a, a_i\in \mathbb{Z}_{\geq 0}$}  and   {\upshape $\gamma_i\in\mbox{Amb}_r$}, $i=1,\cdots, m$, we have

%Let $m\geq0$ and $a_1,\cdots,a_m\geq0$. For any {\upshape $\gamma_1,\cdots,\gamma_m\in\mbox{Amb}_r$}, we have
\ban
\<\prodl_{i=1}^{m}\tau_{a_i}(\gamma_i)\>_{A}=\<\prodl_{i=1}^{m}\tau_{a_i}(\gamma_i)\>_{\sigma(A)}\quad\mbox{for any } A\in H_2(X_r, \mathbb{Z}).
\nan
\end{lemma}

For each $j\in\{1,\cdots,r-1\}$, we define an involution $\sigma_j$ on $H_2(X_r,\bZ)$ by
\ban
\sigma_j(dH-\suml_{k=1}^rd_kE_k)=dH-\suml_{\substack{1\leq k\leq r\\k\neq j,j+1}}d_kE_k-d_{j+1}E_j-d_jE_{j+1}.
\nan
Inspired by \cite[Section 5.1]{GoPa}, we obtain the second type of symmetry as follows.
\begin{lemma}\label{symmetry2}
 For  {\upshape $1\leq j\leq r-1, m, a, a_i\in \mathbb{Z}_{\geq 0}$}  and   {\upshape $\gamma_i\in\mbox{Amb}_r$}, $i=1,\cdots, m$, we have

%Let $m\geq0$ and $a_1,\cdots,a_m\geq0$. For any {\upshape $\gamma_1,\cdots,\gamma_m\in\mbox{Amb}_r$}, we have
\ban
\<\prodl_{i=1}^{m}\tau_{a_i}(\gamma_i)\>_{A}=\<\prodl_{i=1}^{m}\tau_{a_i}(\gamma_i)\>_{\sigma_j(A)}\quad\mbox{for any } A\in H_2(X_r, \mathbb{Z}).
\nan
\end{lemma}
\begin{proof}
Let $X_0$ and $\bar X_0$ be two copies of $\bP^2$, and we fix an isomorphism $\varphi_0:X_0\cong\bar X_0$. Suppose that $p_1,\cdots,p_r\in X_0$ and $\bar p_1,\cdots,\bar p_r\in\bar X_0$ are in general position, such that
\ban
\varphi_0(p_k)=\left\{\begin{array}{cl}
\bar p_k,&k\neq j,j+1,\\
\bar p_{j+1},&k=j,\\
\bar p_j,&k=j+1.
\end{array}\right.
\nan

Let $X_r$ (resp. $\bar X_r$) be the blow-up of $X_0$ at $p_k$ (resp.   $\bar X_0$ at $\bar p_k$), $k=1,\cdots, r$, with $H$ (resp. $\bar H$) being the pullback of a line in $X_0$ (resp. $\bar X_0$) in general position, and $E_k$ (resp. $\bar E_k$) being the exceptional divisor corresponding to $p_k$ (resp. $\bar p_k$).
 For $\gamma=x \mathbbm1_{X_r} +y c_1(X_r) +z [pt]_{X_r}\in \mbox{Amb}(X_r)$, we denote $\bar \gamma=x \mathbbm1_{\bar X_r} +y c_1(\bar X_r) +z [pt]_{\bar X_r}\in \mbox{Amb}(\bar X_r)$. On one hand, by deformation invariance of Gromov-Witten invariants, we have
\ba\label{symmetry21} \<\prodl_{i=1}^{m}\tau_{a_i}(\gamma_i)\>_{dH-\suml_{k=1}^rd_kE_k}^{X_r}=\<\prodl_{i=1}^{m}\tau_{a_i}(\bar \gamma_i)\>_{d\bar H-\suml_{k=1}^rd_k\bar E_k}^{\bar X_r}.
\na
On the other hand,  we have an isomorphism $\varphi_r:X_r\overset{\cong}{\longrightarrow}\bar X_r$ naturally induced from $\varphi_0$. The    induced isomorphisms $\varphi_r^*: H^*(\bar X_r,\mathbb{Z})\to H^*(X_r, \mathbb{Z})$  and  $(\varphi_r)_*: H_*(X_r,\mathbb{Z})\to H_*(\bar X_r, \mathbb{Z})$  satisfy the following properties:
$\varphi_r^*(\bar H)=H,\quad  (\varphi_r)_*(H)=\bar H$
 %\phi_r^*(\bar H)&=&H,\\
$$\varphi_r^*(\bar E_i)=\left\{\begin{array}{cl}
E_i,&i\neq j,j+1,\\
E_{j+1},&i=j,\\
E_j,&i=j+1,
\end{array}\right.\quad\mbox{and}\quad  (\varphi_r)_*(E_i)=\left\{\begin{array}{cl}
\bar E_i,&i\neq j,j+1,\\
\bar E_{j+1},&i=j,\\
\bar E_j,&i=j+1.
\end{array}\right.
$$
\iffalse
\ban
\varphi_r^*(\bar H)&=&H,\\
\varphi_r^*(\bar E_i)&=&\left\{\begin{array}{cl}
E_i,&i\neq j,j+1,\\
E_{j+1},&i=j,\\
E_j,&i=j+1,
\end{array}\right.
\nan
and
\ban
(\varphi_r)_*(H)&=&\bar H,\\
(\varphi_r)_*(E_i)&=&\left\{\begin{array}{cl}
\bar E_i,&i\neq j,j+1,\\
\bar E_{j+1},&i=j,\\
\bar E_j,&i=j+1.
\end{array}\right.
\nan
\fi
Here we remind of our notation convention that divisors are  naturally treated as  (co)homology classes in the corresponding setting. Consequently, we have
$\varphi_r^*(c_1(\bar X_r))=c_1(X_r)$,   which implies $\varphi_r^*(\mathbbm1_{\bar X_r})=\mathbbm1_{X_r}$ and
$\varphi_r^*([pt]_{\bar X_r})=[pt]_{X_r}$. Therefore
\iffalse\ban
\phi_r^*(c_1(\bar X_r))&=&c_1(X_r), \mbox{ which implies } \phi_r^*(\mathbbm1_{\bar X_r})=\mathbbm1_{X_r} \mbox{ and }
\phi_r^*([pt]_{\bar X_r})=[pt]_{X_r};\\
(\phi_r)_*(dH-\suml_{k=1}^rd_kE_k)&=&d\bar H-\suml_{\substack{1\leq k\leq r\\k\neq j,j+1}}d_k\bar E_k-d_{j+1}\bar E_j-d_j\bar E_{j+1}.
\nan
\fi
\ba\label{symmetry22} \<\prodl_{i=1}^{m}\tau_{a_i}(\gamma_i)\>_{dH-\suml_{k=1}^rd_kE_k}^{X_r}&=&\<\prodl_{i=1}^{m}\tau_{a_i}((\varphi_r^*)^{-1}(\gamma_i))\>_{(\varphi_r)_*d(\bar H-\suml_{k=1}^rd_k\bar E_k)}^{\bar X_r}\\
&=&  \<\prodl_{i=1}^{m}\tau_{a_i}(\bar \gamma_i)\>_{d\bar H-\suml_{\substack{1\leq k\leq r\\k\neq j,j+1}}d_k\bar E_k-d_{j+1}\bar E_j-d_j\bar E_{j+1}}^{\bar X_r}.\nonumber
\na
Now the required result follows from \eqref{symmetry21} and \eqref{symmetry22}.
\end{proof}

%We use Lemma \ref{symmetry1} and \ref{symmetry2} to prove the following Lemma \ref{divisortakenout}.

\begin{lemma}\label{divisortakenout} For   {\upshape $m, a, a_i\in \mathbb{Z}_{\geq 0},  \gamma\in \mbox{Prim}_r$}  and   {\upshape $\gamma_i\in\mbox{Amb}_r$}, $i=1,\cdots, m$, we have
 \ban
\suml_{A\in H_2(X_r,\mathbb{Z})}(\int_A\gamma)\<\prodl_{i=1}^{m}\tau_{a_i}(\gamma_i)\>_{A}=0.
\nan
\end{lemma}
\begin{proof}  It suffices to show the statement for  $\gamma=H-E_1-E_2-E_3$ and $\gamma=E_j-E_{j+1}(1\leq j\leq r-1)$, since these classes form   a basis of $\mbox{Prim}_r$.

For $\gamma=H-E_1-E_2-E_3$, by Lemma \ref{symmetry1}, we have
\ban
&&\suml_{A}(\int_A\gamma)\<\prodl_{i=1}^{m}\tau_{a_i}(\gamma_i)\>_{A}\\
&=&\frac12\suml_{A}(\int_A\gamma)\<\prodl_{i=1}^{m}\tau_{a_i}(\gamma_i)\>_{A}+\frac12\suml_{A}(\int_{\sigma(A)}\gamma)\<\prodl_{i=1}^{m}\tau_{a_i}(\gamma_i)\>_{\sigma(A)}\\
&=&\suml_{A}\frac12(\int_{A+\sigma(A)}\gamma)\<\prodl_{i=1}^{m}\tau_{a_i}(\gamma_i)\>_{A}.
\nan
Now the required vanishing result follows from $\int_{A+\sigma(A)}(H-E_1-E_2-E_3)=0$.
%%\ban \int_{A+\sigma(A)}(H-E_1-E_2-E_3)=0. \nan

For $\gamma=E_j-E_{j+1}(1\leq j\leq r-1)$, by Lemma \ref{symmetry2}, we have
\ban
&&\suml_{A}(\int_A\gamma)\<\prodl_{i=1}^{m}\tau_{a_i}(\gamma_i)\>_{A}\\
&=&\frac12\suml_{A}(\int_A\gamma)\<\prodl_{i=1}^{m}\tau_{a_i}(\gamma_i)\>_{A}+\frac12\suml_{A}(\int_{\sigma_j(A)}\gamma)\<\prodl_{i=1}^{m}\tau_{a_i}(\gamma_i)\>_{\sigma_j(A)}\\
&=&\suml_{A}\frac12(\int_{A+\sigma_j(A)}\gamma)\<\prodl_{i=1}^{m}\tau_{a_i}(\gamma_i)\>_{A}.
\nan
Now the required vanishing result follows from $\int_{A+\sigma_j(A)}(E_j-E_{j+1})=0$.
%%\ban \int_{A+\sigma_j(A)}(E_j-E_{j+1})=0. \nan
\end{proof}

%We will use induction to prove Theorem \ref{vanishingproperty}. The initial case is treated in the following Lemma \ref{initial}.

\begin{prop}\label{initial}
Let   {\upshape $m, a, a_i\in \mathbb{Z}_{\geq 0},  \gamma\in \mbox{Prim}_r$}  and   {\upshape $\gamma_i\in\mbox{Amb}_r$}, $i=1,\cdots, m$.
If $\gamma_i\in H^{>0}(X_r)$ whenever $a_i\neq 0$, then we have
 \ban
\suml_{A\in H_2(X_r,\mathbb{Z})}\<\tau_0(\gamma)\prodl_{i=1}^{m}\tau_{a_i}(\gamma_i)\>_{A}=0.
\nan
\end{prop}
\begin{proof} It suffices to show the vanishing of $\suml_{A\neq0}\<\tau_0(\gamma)\prodl_{i=1}^{m}\tau_{a_i}(\gamma_i)\>_{A}$ due to Lemma \ref{degreezero}. We first assume $\gamma_i\in H^{>0}(X_r)$ for all $i$, so that    $\gamma\cup \gamma_i=0$ always holds. Hence, we have
 \ban
\suml_{A\neq0}\<\tau_0(\gamma)\prodl_{i=1}^{m}\tau_{a_i}(\gamma_i)\>_{A}=\suml_{A\neq 0}(\int_A\gamma)\<\prodl_{i=1}^{m}\tau_{a_i}(\gamma_i)\>_{A}=0.
\nan
 by using the Divisor Axiom and Lemma \ref{divisortakenout}.
Then the general statement follows immediately from the Fundamental Class Axiom.
\end{proof}

%We use induction to prove the following Lemma \ref{positivedegree}, of which the initial case was treated in Lemma \ref{initial}.

\begin{prop}\label{positivedegree}Let   {\upshape $m, a, a_i\in \mathbb{Z}_{\geq 0},  \gamma\in \mbox{Prim}_r$}  and   {\upshape $\gamma_i\in\mbox{Amb}_r$}, $i=1,\cdots, m$.
If $\gamma_i\in H^{>0}(X_r)$ for all $i$, then we have
 \ban
\suml_{A\in H_2(X_r, \mathbb{Z})}\<\tau_a(\gamma)\prodl_{i=1}^{m}\tau_{a_i}(\gamma_i)\>_{A}=0.
\nan
\end{prop}
\begin{proof}
We use induction on $a$. The case $a=0$ is done in Proposition \ref{initial}. Now assume that the case $a=a'\geq0$ is verified.  For $a=a'+1$, there are the following three cases.
\begin{itemize}
\item[(i)] $\underline{m\geq2}$.
 Extend $\{\phi_1=\mathbbm1, \phi_2=c_1, \phi_3=[pt]\}$ to a basis $\{\phi_\alpha\}_{\alpha}$ of $H^*(X_r)$ such that $\phi_\alpha\in\mbox{Prim}_r$ for $\alpha>3$, and let $\{\phi^\alpha\}_{\alpha}$ be its dual basis.  Using TRR, we conclude that
the  quantity $\suml_A\<\tau_{a'+1}(\gamma)\prodl_{i=1}^{m}\tau_{a_i}(\gamma_i)\>_{A}$ is equal to
  $$
 \suml_{I} \suml_\alpha\Big(\suml_A\<\tau_0(\phi_\alpha)\tau_{a'}(\gamma)\prodl_{i\in I}\tau_{a_i}(\gamma_i)\>_{A}\Big)
 \cdot\Big(\suml_A\<\tau_0(\phi^\alpha)\tau_{a_1}(\gamma_1)\tau_{a_2}(\gamma_2)\hspace{-0.45cm}\prodl_{j\in \{3,\cdots,m\}\setminus I}\hspace{-0.45cm}\tau_{a_j}(\gamma_j)\>_{A}\Big),
$$
where the first summation is over subsets $I\subset \{3,\cdots, m\}$.
  For $\alpha\leq 3$,   the first big parentheses on RHS is zero by induction (together with the Fundamental Class Axiom when $\alpha=1$); for $\alpha>3$, we note  $\phi^\alpha\in\mbox{Prim}_r$, and hence the second big parentheses on RHS is zero by Proposition  \ref{initial}. So LHS is vanishing.

\item[(ii)] $\underline{m=1}$. By   (i), Lemma \ref{degreezero}, the Divisor Axiom and the Degree Axiom, we have
$$0=\suml_A\<\tau_{a'+1}(\gamma)\tau_{a_1}(c_1^2)\tau_0(c_1)\>_A
 = (a'+a_1+{\deg\gamma\over 2}+2)\suml_{A}\<\tau_{a'+1}(\gamma)\tau_{a_1}(c_1^2)\>_A.$$
Here we have assumed $\gamma$ to be homogenous without loss of generality. Similarly,
 \begin{align*}
0&= \suml_A\<\tau_{a'+1}(\gamma)\tau_{a_1}(c_1)\tau_0(c_1)\>_A\\
&= (a'+a_1+{\deg\gamma\over 2}+1)\suml_{A}\<\tau_{a'+1}(\gamma)\tau_{a_1}(c_1)\>_A+\suml_{A}\<\tau_{a'+1}(\gamma)\tau_{a_1-1}(c_1^2)\>_A
\end{align*}
It follows that $\suml_{A}\<\tau_{a'+1}(\gamma)\tau_{a_1}(\gamma_1)\>_A=0$ whenever $\gamma_1\in \mbox{Amb}_r\cap H^{>0}(X_r)$.
\item[(iii)] $\underline{m=0}$. By using (ii), Lemma \ref{degreezero}, the Divisor Axiom and the Degree Axiom, we are done:
\ban
0=\suml_A\<\tau_{a'+1}(\gamma)\tau_0(c_1)\>_A=(a'+{\deg\gamma\over 2}+1)\suml_{A}\<\tau_{a'+1}(\gamma)\>_A.
\nan
\end{itemize}
This proves the case $a=a'+1$.
\end{proof}
\bigskip

\begin{proof}[Proof of Theorem \ref{vanishingproperty}] By Proposition \ref{positivedegree}, it remains to show that for any $m_0, m_1\geq 0$,
   \ban
\suml_A\<\tau_a(\gamma)\prodl_{i=1}^{m_0}\tau_{a_{i}}(\mathbbm1)\prodl_{j=1}^{m_1}\tau_{b_{j}}(\gamma_j)\>_{A}=0
\nan
\noindent holds for any  $a,a_i, b_j\geq0$,   $\gamma\in\mbox{Prim}_r$ and   $\gamma_j\in\mbox{Amb}_r\cap H^{>0}(X_r)$, where $1\leq i\leq m_0$ and $1\leq j\leq m_1$.

 We use induction on $m_0$.
 The case $m_0=0$ is proved in Proposition \ref{positivedegree}. Now assume that the cases $m_0\leq m_0'$ are verified. %, and we further use induction on $\suml_{i=1}^{m_0}a_{i}$ to prove the case
 For $m_0=m_0'+1$,  by Lemma \ref{degreezero} and the Fundamental Class Axiom, we can assume  $a_{i}\geq1$ for all $i$. There are    two cases as follows.
\begin{itemize}
\item[(i)] $\underline{m_1>0}$. We use the same (dual)  basis $\{\phi_\alpha\}_\alpha$ (resp. $\{\phi^\alpha\}_{\alpha}$)  as in case (i) in the proof of Proposition \ref{positivedegree}. Using TRR, we have
\ban
&&\suml_A\<\tau_a(\gamma)\prodl_{i=1}^{m'_0+1}\tau_{a_{i}}(\mathbbm1)\prodl_{i=1}^{m_1}\tau_{b_{i}}(\gamma_i)\>_{A}\\
&=&\suml_{\substack{I_0\sqcup J_0=\{2,\cdots,m'_0+1\}\\I_1\sqcup J_1=\{2,\cdots,m_1\}}}\suml_\alpha\bigg(\suml_A\<\tau_0(\phi_\alpha)\tau_{a_1-1}(\mathbbm1)\prodl_{i\in I_0}\tau_{a_i}(\mathbbm1)\prodl_{i\in I_1}\tau_{b_i}(\gamma_i)\>_A\bigg)\\
&&\qquad\cdot\bigg(\suml_A\<\tau_0(\phi^\alpha)\tau_a(\gamma)\tau_{b_1}(\gamma_1)\prodl_{j\in J_0}\tau_{a_j}(\mathbbm1)\prodl_{j\in J_1}\tau_{b_j}(\gamma_j)\>_A\bigg).
\nan
If $\phi_\alpha\in\mbox{Prim}_r$, then the first big parentheses on RHS is zero by induction on $\suml_{i=1}^{m_0}a_{i}$; if $\phi_\alpha\in\mbox{Amb}_r$, then the second big parentheses on RHS is zero by induction on $m_0$. So LHS is vanishing.

\item[(ii)] $\underline{m_1=0}$. We use Lemma \ref{degreezero}, the Divisor Axiom and the Degree Axiom to get
\ban
&&\suml_A\<\tau_a(\gamma)\prodl_{i=1}^{m'_0+1}\tau_{a_i}(\mathbbm1)\cdot\tau_0(c_1)\>_A\\
&=&(a+{\deg\gamma\over 2}+\suml_{i=1}^{m'_0+1}(a_i-1))\suml_A\<\tau_a(\gamma)\prodl_{i=1}^{m'_0+1}\tau_{a_i}(\mathbbm1)\>_A\\
&&\qquad+\suml_{i=1}^{m'_0+1}\suml_A\<\tau_a(\gamma)\tau_{a_i-1}(c_1)\prodl_{\substack{1\leq j\leq m'_0+1\\j\neq i}}\tau_{a_j}(\mathbbm1)\>_A.
\nan
Note that LHS is zero by case (i), and the second term on RHS is zero by induction on $m_0$. This gives the required vanishing result.
\end{itemize}
This proves the case $m_0=m'_0+1$.
\end{proof}

\section{Conjecture $\cO$ for del Pezzo surfaces}
%%\section{Conjecture $\mathcal{O}$}
In this section, we will prove Theorem \ref{thmconjO}, namely the  conjecture $\cO$ for del Pezzo surfaces, by using the generalized Perron-Frobenius theorem in section 3.1.
\noindent We will just   discuss $X_r$ with $1\leq r\leq 8$,  as it has been known for   $\bP^1\times \bP^1$ and $\bP^2$. We remark that Theorem \ref{thmconjO} could also be directly proved by analysing the characteristic polynomial of $\hat c_1$ which was studied in \cite{BaMa} based on  explicit descriptions of the relevant Gromov-Witten invariants \cite{GoPa}. However, we expect our method to have further applications for other Fano manifolds. As we will see,  our proof will only require information on    parts of the Gromov-Witten invariants.

\subsection{Proof of Theorem \ref{thmconjO}}
Eigenvalues of $\hat c_1$ on $QH^\bullet(X_r)$ coincide with that of the matrix  of $\hat c_1$ with respect to any choice of   bases of $H^*(X_r)=H^\bullet(X_r)$. We take the $\bZ$-basis as in Examples \ref{matX1} and \ref{matX2} if $r\in \{1, 2\}$, and take the $\mathbb{Q}$-basis $\{\mathbbm1, c_1, E_1,\cdots, E_r, [pt]\}$ of $H^*(X_r)$ for $3\leq r\leq 8$. The corresponding matrices $M_r$ have been explicitly described in section \ref{QHdelPezzo} for $1\leq r\leq 3$.
Let us achieve our aim by assuming the next proposition first.

\begin{prop}\label{propMatrixdelPezzo}
 Write   $M_r=\big(m_{ij}\big)$ and $M_r^2=\big(m_{ij}^{(2)}\big)$. For    $3\leq r\leq 8$, we have
 \begin{enumerate}
  \item $\sum_{i=1}^{r+3}m_{ij}> 0$ for any     $j\in \{1, \cdots, r+3\}$, and $m_{2j}\geq  0$ for any $j$;
  \item   $m_{ij}^{(2)}\geq 0$   for any    $i, j$, and   $m_{ij}^{(2)}> 0$ holds if
      $i\in\{1, 2, r+3\}$;
  \item     $m_{22}^{(2)}>\sum_{k=3}^{r+2} m_{ik}^{(2)}$ if $3\leq i\leq r+2$;
\end{enumerate}
 \end{prop}

Now we set $P:=(a-1) E_{22}+\sum_{k=1}^{r+3}E_{kk}+ b\sum_{k=3}^{r+2}E_{k2}$ where $a, b>0$ and $E_{ij}$ denotes the $(r+3)\times (r+3)$ matrix whose entries are all zero but the $(i,j)$-entry given by $1$. It follows that
$P^{-1}=(c-1) E_{22}+\sum_{k=1}^{r+3}E_{kk}+ d\sum_{k=3}^{r+2}E_{k2}$ with   $d=-{b\over a}=-bc$.
Write $PM_rP^{-1}=\hat M_r=\big(\hat m_{ij}\big)$ and $ \hat M_r^2=\big(\hat m_{ij}^{2}\big)$. By direct calculations, we have
  $$\hat m_{ij}=\begin{cases}
     m_{ij},&\mbox{if } i=1, 3 \mbox{ and } j\neq 2,\\
     a m_{2j},&\mbox{if } i=2 \mbox{ and } j\neq 2,\\
     m_{ij}+b m_{2j}, &\mbox{if } 3\leq i\leq r+2\mbox{ and }j\neq 2
     \end{cases}\,\, $$
    and  $$ \hat m_{i2}=\begin{cases}
     c m_{i2}+ d \sum_{k=3}^{r+2} m_{ik},&\mbox{if } i=1, 3,\\
       ac m_{22}+d a \sum_{k=3}^{r+2} m_{2k},&\mbox{if } i=2,\\
     c m_{i2}+c bm_{22}+d \sum_{k=3}^{r+2} (m_{ik}+bm_{i, k}), &\mbox{if } 3\leq i\leq r+2.
  \end{cases}$$
Clearly, $M_r$ and $\hat M_r$ have same eigenvalues for any $a\neq 0$. In particular, we can choose sufficiently small positive numbers $b, {b\over c}$ with $c<1$. It follows that
  $a>1$, and \begin{align*}
 \sum_{i=1}^{r+3}\hat m_{ij}&\geq \sum_{i=1}^{r+3}m_{ij}+b\sum_{k=3}^{r+2}m_{2j}>0 \mbox{ if } j\neq 2;\\
    \sum_{i=1}^{r+3}\hat m_{i2}&\geq c\big(\sum_{i=1}^{r+3}m_{i2}+b rm_{22}- (b+{b\over c}+b+b^2)\sum_{i=1}^{r+3} \sum_{k=3}^{r+2}|m_{ik}|\big)>0.\end{align*}
The expression of $\hat m_{ij}^{(2)}$ can be read off directly from that of $\hat m_{ij}$ by replacing $m_{ij}$ with $m_{ij}^{(2)}$. It follows that   $\hat m_{ij}^{(2)}>0$ for any $i, j$. (Indeed for $3\leq i\leq r+2$, $\hat m_{i2}^{(2)}= c m_{i2}^{(2)}+cb (m_{22}^{(2)}-\sum_{k=3}^{r+2}m_{ik}^{(2)})-b\sum_{k=3}^{r+2}m_{ik}^{(2)}>0$, as $c m_{i2}^{(2)}\geq 0, cb>0, m_{22}^{(2)}-\sum_{k=3}^{r+2}m_{ik}^{(2)}>0$ and $b>0$ is sufficiently small. Arguments for the remaining cases are similar and easier.)
Applying Theorem \ref{genPFthm} for $\hat M$, we conclude that $\rho(\hat c_1)=\rho(M_r)=\rho(\hat M)$ is an eigenvalue of $\hat c_1$ of multiplicity one. Moreover, it is the unique eigenvalue of $\hat c_1$ with modulus $\rho(\hat M)$ by Proposition \ref{propuniquerho}. That is, $X_r$ satisfies Property $\cO$.

\subsection{Proof of Proposition \ref{propMatrixdelPezzo}} By Example \ref{matX3}, Proposition \ref{propMatrixdelPezzo} holds for   $M_3$.
   Therefore we just consider $4\leq r\leq 8$ in this subsection. The     dual  basis of  $[\phi_1,\cdots, \phi_{r+3}]:=[\mathbbm1, c_1, E_1,\cdots, E_r, [pt]]$  in $H^*(X_r)$is given by
$$[\mathbbm1^\vee, c_1^\vee, E_1^\vee, \cdots, E_r^\vee, [pt]^\vee]=[[pt], {H\over 3}, {H\over 3}-E_1, \cdots, {H\over 3}-E_r, \mathbbm1].$$

    \iffalse Let us review the following    well-known properties, where $A\in H_2(X_r, \mathbb{Z})$ is nonzero,  each  $\alpha_i\in H^*(X_r)$ is of homogeneous degree and $m\geq 0$. We simply denote $\langle \cdots\rangle_{A}:=\langle\cdots\rangle_{0, A}$ for convenience.

\noindent\textbf{Degree Axiom} $\langle \alpha_1, \cdots, \alpha_{m+1}\rangle_{A}=0$ unless $\sum_{i=1}^{m+1}\deg \alpha_i= m+\int_A c_1$.

\noindent\textbf{Divisor Axiom} $\langle \alpha_1, \cdots, \alpha_m,   \gamma\rangle_{A}= \< \alpha_1, \cdots, \alpha_m\>_{A}\int_A \gamma$; when $m=0$, the right hand side is read off as  $\int_A\gamma$ if $\int_Ac_1=1$ and $A$ is effective, or $0$ otherwise.

\noindent\textbf{Fundamental Class Axiom}  $\langle \alpha_1, \cdots, \alpha_m, \mathbbm1\rangle_{A}=0$.
\fi
We will prove Proposition \ref{propMatrixdelPezzo} by showing that the matrix $M_r$
   is in fact of the  following form with required properties
  \begin{equation}\label{matrixMr}
     M_r=\left[\begin{array}{ccccccc}
0&m_{12}&m_{13}&m_{14}&\cdots & m_{1,r+2}&m_{1,r+3}\\
1&m_{22}&m_{23}&m_{24}&\cdots &m_{2, r+2}&m_{2, r+3}\\
0&0& d_r &0&\cdots &0&0\\
0&0&0& d_r&\ddots  & \vdots  &0\\
\vdots&\vdots&0&\ddots&\ddots &0&0\\
0&0&0&\cdots &0& d_r&\vdots \\
0&9-r&1&1&\cdots &1&0
\end{array}\right].
  \end{equation}

\iffalse with the following properties:
\begin{enumerate}
  \item $m_{ij}> 0$ for any  $i\in \{1, 2\}$ and $j\in \{2, \cdots, r+3\}$;
  \item $m_{1j}+d_r>0$ and $(9-r)m_{2j}+d_r>0$ for any $j\in \{3,\cdots, r+2\}$;
  \item  $\sum_{k=1}^{r+3}m_{ik}m_{kj}\geq 0$  for any $i\in\{1, 2\}$ and  any   $j\in \{3,\cdots, r+2\}$.
\end{enumerate}
\fi

 Let us start with simple calculations.   Entries $m_{ij}$ of $M_r$ are   genus zero Gromov-Witten invariants.
  Clearly, we have  $m_{i1}=\delta_{i, 2}$ since  $\mathbbm{1}$ is the identity element in $QH^*(X_r)$.
 For the last row, we have $m_{r+3, j}=\sum_{A\in H_2(X_r, \bZ)}\<c_1, \phi_j, \mathbbm1\>_{A}$
 and hence $m_{r+3, j}= \<c_1, \phi_j, \mathbbm1\>_{0}$ by the Fundamental Class Axiom, namely it is given by the coefficient of $[pt]$ in the classical cup product $c_1\cup \phi_i$.  Consequently, we have $m_{r+3, 1}=m_{r+3, r+3}=0$ for the degree reason, $m_{r+3, 2}=c_1\cdot c_1=9-r$, and $m_{r+3, j}=c_1\cdot E_{j-2}=1$ for $3\leq j\leq r+2$.

 A smooth rational curve $E$ of $X_r$ is called exceptional if $E\cdot E=-1$ and $c_1\cdot E=1$.
 A genus zero Gromov-Witten invariant $\<\cdots \>_{A}$ for $X_r$ is nonzero only if $A\in H_2(X_r, \mathbb{Z})$ is effective, which can be characterized in terms of effective divisors as follows.

 \begin{prop}[Corollary 3.3 of \cite{BaPo}] The semigroup of classes of effective divisors on $X_r$   is generated by  the classes of exceptional curves if $4\leq r\leq 7$ and by the classes of exceptional curves together with $c_1$ for $r=8$.
 \end{prop}
 \begin{prop}[\cite{Manin}] \label{propexceptionalcurve}The classes of exceptional curves are precisely as follows.
 \begin{enumerate}
   \item $E_i$, $1\leq i\leq r$;
   \item $H-E_i-E_j$, $1\leq i<j\leq r$;
   \item $2H-E_{i_1}-\cdots -E_{i_5}$,\quad $1\leq i_1<i_2<\cdots<i_5\leq r$;
   \item $3H-E_k-\sum_{j=1}^7E_{i_j}$,\quad $1\leq i_1<i_2<\cdots<i_7\leq r$ and $k\in\{i_1, \cdots, i_7\}$;
   \item $4H-\sum_{j=1}^8E_j- E_{i_1}-E_{i_2}-E_{i_3}$,\quad   $1\leq i_1<i_2<  i_3\leq 8=r$;
   \item $5H-\sum_{j=1}^82E_j+E_{i_1}+E_{i_2}$,\quad   $1\leq i_1<i_2\leq 8=r$;
   \item $6H-\sum_{j=1}^82E_j- E_{i}$,\quad   $1\leq i \leq 8=r$.
 \end{enumerate}
 \end{prop}

 For convenience, we will simply denote a summation $\sum_{A~:~ A\in H_2(X_r, \mathbb{Z}){\rm \,\,is\,\, effection }; ***}$ by $\sum_{***}$ in the rest of this subsection. The next lemma ensures that the summation to be taken does contain one term of positive Gromov-Witten invariants. The first statement holds since Gromov-Witten invariants for $X_r$ are enumerative \cite{GoPa}; the second statement can be easily deduced from the results therein (or by using Theorems 1.2 and 1.4 of \cite{Hu}).
 \begin{lemma} \label{lemmanonempty}\begin{enumerate}
   \item For any $A\in H_2(X_r, \mathbb{Z})$, we have $\langle [pt]\rangle_{A}\geq 0,  \<[pt], [pt]\>_{A}\geq 0$.
   \item $\langle [pt]\rangle_{H-E_1-E_2}=\<[pt]\>_{H-E_1}=\<[pt], [pt]\>_{H}=1$; $\<[pt]\>_{E_i}=0$ for $1\leq i\leq r$.
 \end{enumerate}
 \end{lemma}
\noindent  Here we notice $\int_{H-E_1-E_2}c_1=1, \int_{H-E_1}c_1=2$ and $\int_Hc_1=3$.

  \begin{lemma}\label{lemcoeffc1}
   $c_1\bullet c_1= m_{12}\mathbbm1+m_{22} c_1+(9-r)[pt]$ where $m_{12}=\sum_{\int_A c_1=2}4\<[pt]\>_A>0$ and $m_{22}=\sum_{\int_A c_1 =1}\int_A {H\over 3}>0$.
 \end{lemma}
 \begin{proof}Recall the notation $\alpha\bullet \beta=\alpha\star\beta|_{\mathbf{q}=\mathbf{1}}$.
 Noting $\mathbbm1^\vee=[pt]\in H^4(X_r)$, we have
   $m_{21}=\sum_{A~:~ \int_Ac_1 =2}\langle c_1, c_1, \mathbbm1^\vee\rangle_A$ by the Degree Axiom, and hence
   $m_{21}=\!\!\!\sum\limits_{\int_Ac_1=2}\!\!\!4\langle [pt]\rangle_A>0$ by the Divisor Axiom and Lemma \ref{lemmanonempty}.

   The computation for $m_{22}$ is the same. The quantity $m_{r+3, 2}=9-r$  has been discussed above. Since $c_1\cup E_j^\vee=\big(\int_{[X_r]}c_1\cup E_j^\vee\big)[pt]=0$, we have $E_j^\vee\in \mbox{Prim}_r$, and consequently $m_{2, 2+j}=\sum_A\langle c_1, c_1, E_j^\vee\rangle_A=0$ for $1\leq j\leq r$ by Proposition \ref{vanishingproperty}.
 \end{proof}

\begin{lemma}\label{lemmafirsttworows} For any $1\leq j\leq r$, we have
 \begin{enumerate}
   \item  $m_{1, j+2}=\!\!\!\sum\limits_{\int_A c_1=2}\!\!2\big(\int_A E_j\big)\<[pt]\>_A$, \qquad $m_{1, r+3}=\sum_{\int_Ac_1=3}3\<[pt],[pt]\>_A$;
    \item   $m_{2, j+2}=\!\!\!\sum\limits_{\int_A c_1=1}\!\!\big(\int_A E_j\big)\big(\int_A{H\over 3}\big)$,\qquad $ m_{2, r+3}=\sum_{\int_Ac_1=2}2\big(\int_A{H\over 3}\big)\<[pt]\>_A.$
  \end{enumerate}
 Furthermore, they are all positive.
 \end{lemma}
 \begin{proof}
 Calculations for  $m_{ij}$ in the statement   are the same as that for $m_{21}$ in Lemma \ref{lemcoeffc1}.

 Write $A\in H_2(X_r, \mathbb{Z})$ as $A=a_0H-\sum_{i=1}^ra_iE_i$. Every effective curve class  $A$ is a nonnegative combination of classes of   effective curve classes (together with $c_1$ if  $r=8$).   The hypotheses $\int_Ac_1=2$ and $\<[pt]\>_A\neq 0$ will imply that   $a_0>0$ and $a_i\geq 0$ for $1\leq i\leq r$, in which case we have
  $\int_A E_j\geq 0$. Therefore we conclude $m_{1, j+2}>0$ for $1\leq j\leq r$ by Lemma \ref{lemmanonempty}. Arguments for the remaining cases are similar.
 \end{proof}

 \begin{lemma}\label{lemdiagnoal}
   For any $1\leq i\leq r$, we have $m_{2+i, 2+i}=d_r$, where $d_r:=-3,-4,-6,-12,-60$ for $r=4,5,6,7,8$ respectively.
 \end{lemma}

\begin{proof}
We have $m_{2+i, 2+i}=\sum_{\int_A c_1=1}\big(\int_{A}E_i\big)\big(\int_A{H\over 3}-E_i\big)$  by definition and the Divisor Axiom.
  Without loss of generality, we can assume $i=r$. Since $\int_A c_1=1$, the summation can be reduced to those over the classes of exception curves. We discuss all the possibilities with respect to the cases in Proposition \ref{propexceptionalcurve}.
  \begin{enumerate}
    \item The only nonzero contribution is given by $A=E_r$, and  $\int_AE_r\int_A({H\over 3}-E_r)=-1$.
    \item The nonzero contributions come from those $H-E_{i_1}-E_r$ with $1\leq i_1<r$, each of which   contributes a same quantity; therefore
    the total contribution is given by $(r-1)\int_{H-E_1-E_r}E_r\big(\int_{H-E_1-E_r}{H\over 3}-E_r\big)=-{2(r-1)\over 3}$.
    \item Here $r\geq 5$. The total nonzero contribution is given by
        $$C_{r-1}^4\cdot \big(\int_{2H-E_1-E_2-E_3-E_4-E_r}E_r\big)\big(\int_{2H-E_1-E_2-E_3-E_4-E_r}{H\over 3}-E_r\big)=-{1\over 3}C_{r-1}^4.$$
    \item Here $r\geq 7$. The total nonzero contribution is as follows. (We notice that classes of the form  $3H-2E_1-\sum_{j=2}^6E_{j}-E_r$ does not make nonzero contributions.)
             $$C_{r-1}^6\big(\int_{3H-\sum_{j=1}^6E_{j}-2E_r}E_r\big)\big(\int_{3H-\sum_{j=1}^6E_{j}-2E_r}{H\over 3}-E_r\big)=-2C_{r-1}^6.$$
    \item Here $r=8$. Denote  $A_1:=4H-\sum_{j=1}^8E_j- E_{1}-E_{2}-E_{8}$ and $A_2:=4H-\sum_{j=1}^8E_j- E_{1}-E_{2}-E_{3}$. The total nonzero contribution is given by
         $$C_7^2\big(\int_{A_1}E_8\big)\big(\int_{A_1}{H\over 3}-E_8\big)+C_7^3\big(\int_{A_2}E_8\big)\big(\int_{A_2}{H\over 3}-E_8\big)=-{49\over 3}.$$
    \item Here $r=8$.   Denote  $A_1:=5H-\sum_{j=1}^82E_j+E_{1}+E_{8}$ and $A_2:=5H-\sum_{j=1}^82E_j+ E_{1}+E_{2}$. The total nonzero contribution is given by
         $$C_7^1\big(\int_{A_1}E_8\big)\big(\int_{A_1}{H\over 3}-E_8\big)+C_7^2\big(\int_{A_2}E_8\big)\big(\int_{A_2}{H\over 3}-E_8\big)=-{28\over 3}.$$

    \item Here $r=8$. There is  only one nonzero contribution given by
    $$\big(\int_{6H-\sum_{j=1}^82E_{j}-E_r}E_r\big)\big(\int_{6H-\sum_{j=1}^8E_{j}-E_r}{H\over 3}-E_r\big)=-3.$$
  \end{enumerate}
 We should have also considered  $A=c_1$ when $r=8$, while it does not make contributions since $\int_{c_1}{H\over 3}-E_r=0$.

Taking the summation of all the nonzero contributions deduces the required result.
   \end{proof}
\begin{remark}\label{rmkquantumPeriod}
   The number $d_r$ is related with the quantum period $G_{X_r}(t)$ of  $X_r$ ($r\leq 4\leq 8$) computed in \cite[section G]{CCGK} in the way: $G_{X_r}(t)=e^{d_rt}(1+O(t))$.
\end{remark}

   \begin{lemma}
  For any $3\leq i\leq r+2$ and $3\leq j\leq r+3$,  $m_{i, j}=0$   if $i\neq j$.
  \end{lemma}
 \begin{proof}
    By  \ref{vanishingproperty}, we have $m_{i, r+3}=\sum_A\< c_1, [pt], E_{i-2}^\vee\>_A=0$ as  $c_1, [pt]\in \mbox{Amb}_r$ and $E_{i-2}^\vee\in \mbox{Prim}_r$. The vanishing  $m_{i1}=m_{i2}=0$ have been shown earlier.

Now we compute   $m_{2+i, 2+j}=\sum_{\int_A c_1=1}\big(\int_{A}E_i\big)\big(\int_A{H\over 3}-E_j\big)$ for $1\leq i, j\leq r$ with $i\neq j$. It suffices to deal with the case when $(i, j)=(1, 2)$.  We discuss all the cases as in the proof of Lemma \ref{lemdiagnoal}.
 Clearly, there is no contribution in case (1) or when $A=c_1$ and $r=8$. Denote by $C(A_1, A_2)$ the contributions to $m_{3, 4}$
 from classes of the form $A_1$ or $A_2$. We mean   $C_a^b=0$ if $b>a$. We have
 \begin{center}
  \begin{tabular}{|c|c|c|c|c|}
   \hline
   Cases & $A_1$ &  $A_2$  &  $C(A_1, A_2)$ &    \\ \hline
     (2) & $a_1=a_2=1$  & $a_1=1, a_2=0$  & $ (-{2\over 3})+(r-2) \cdot {1\over 3}$  & $r\geq 4$  \\ \hline
     (3) & $a_1=a_2=1$  & $a_1=1, a_2=0$ & $C_{r-2}^3 \cdot (-{1\over 3})+C_{r-2}^4 \cdot {2\over 3}$  & $r\geq 5$  \\      \hline
     (4) & $a_1=1, a_2=2$ & $a_2=0$  & $C_{r-2}^5 \cdot (-1)+C_{r-1}^7\cdot r$ & $r\geq 7$  \\      \hline
     (5) & $a_1=a_2=1$  & $a_1=a_2=2$  & $C_{6}^3 \cdot ({1\over 3})+C_{6}^1 \cdot {-4\over 3}$  & $r=8$  \\      \hline
     (5) & $a_1=1, a_2=2$  & $a_1=2, a_2=1$  & $C_{6}^2 \cdot 1\cdot ({-2\over 3})+C_{6}^2 \cdot 2\cdot {1\over 3}$  & $r=8$  \\      \hline
     (6) & $a_1=a_2=1$  & $a_1=a_2=2$  & $ {2\over 3}+C_{6}^4 \cdot {-2\over 3}$  & $r=8$  \\      \hline
     (6) & $a_1=1, a_2=2$  & $a_1=2, a_2=1$  & $C_{6}^5  \cdot ({-1\over 3})+C_{6}^5 \cdot   {4\over 3}$  & $r=8$  \\      \hline
     (7) & \multicolumn{2}{c|}{$6H-\sum_{k=1}^82E_k-E_2$}& $-2$  & $r=8$  \\      \hline
 \end{tabular}
 \end{center}
By direct calculation for each $r$, the summation of all    $C(A_1, A_2)$ equals $0$.
 \end{proof}

  \begin{lemma}
      $m_{1,j+2}+d_r>0$ and $(9-r)m_{2,j+2}+d_r>0$ for any $j\in \{1,\cdots, r\}$.
  \end{lemma}
  \begin{proof}

    For  $A=2H-E_j-E_{i_1}-E_{i_2}-E_{i_3}$ where $1\leq i_1<i_2<i_3\leq r$ are distinct with $j$, we have $\int_Ac_1=2, \int_A E_j=1$ and $\<[pt]\>_A=1$ (namely there exists a unique conic passing through $5$ points in $X_r$). Moreover, we note
     $\int_{H-E_j}c_1=2, \int_{H-E_j}E_j=1$ and $\<[pt]\>_{H-E_j}=1$ (namely there exists a unique line passing through $2$ points in $X_r$).
  Hence $$m_{1, j+2}=\sum_{\int_Ac_1=2} 2\int_AE_j\<[pt]\>_A \geq2\cdot 1\cdot 1+ C_{r-1}^3\cdot 2\cdot 1\cdot 1=2+2C_{r-1}^3>-d_r.$$

 Write $A=a_0H-\sum_{i=1}^8a_iE_i$. We may assume $j\neq 1$. By Lemmas \ref{lemmafirsttworows},\,\ref{lemdiagnoal}, we  have
   \begin{align*}
     (9-r)m_{2,j+2}+h_r&=(9-r)\sum_{\int_{A}c_1=1} \big(\int_{A}E_j\big)\big(\int_A{H\over 3}\big)+ \sum_{\int_{A}c_1=1} \big(\int_{A}E_j\big)\big(\int_A{H\over 3}-E_j\big)\\
            &=-1+(r-1)  \big(\int_{H-E_1-E_j}E_j\big)\big(\int_{H-E_1-E_j}{(10-r)\over 3}H-E_j\big)\\
             &\quad + \sum_{a_0>1; \int_{A}c_1=1} \big(\int_{A}E_j\big)\big(\int_A{(9-r)+1\over 3}H-E_j\big)\\
             &={(7-r)(r-1)-3\over 3}+\sum_{a_0>1; \int_Ac_1=1}(a_j)({10-r\over 3}a_0-a_j)
   \end{align*}
For $4\leq r\leq 7$, both parts are positive. For $r=8$, the second part is much larger than ${10/3}$. Therefore  the above quantity is always positive.
\end{proof}

\begin{cor}\label{corcondone}
  For any $j\in\{1, \cdots, r+3\}$, we have $\sum_{i=1}^{r+3}m_{ij}>0$ and $m_{r+3, j}^{(2)}>0$.
\end{cor}
\begin{lemma}
  For $1\leq i\leq r+2$, we have $m_{ij}^{(2)}\geq 0$ for any $j$. Furthermore if $i\in \{1, 2\}$, then $m_{ij}^{(2)}>0$ for any $j$.
\end{lemma}
\begin{proof}
So far we have shown that $M_r$ is of   form   \eqref{matrixMr}. Due to Lemma \ref{lemmafirsttworows}, the statement obviously holds except for $m_{ij}^{(2)}$
  with  $i\in \{1, 2\}$ and  $j\in \{3,\cdots, r+2\}$.
 Recall   the basis $[\phi_1, \cdots, \phi_{r+3}]=[\mathbbm1, c_1, E_1,\cdots, E_r, [pt]]$. By definition, the entry $m_{ij}$ in the matrix $M_r$ is the coefficient of $\phi_i$ in the product $c_1\bullet \phi_j$, namely   $m_{ij}=\sum_{A}\<c_1, \phi^i, \phi_j\>_A$ where
   $[\phi^1, \cdots, \phi^{r+3}]=[[pt], {H\over 3}, {H\over 3}-E_1,\cdots, {H\over 3}-E_r, \mathbbm1]$ are the dual basis of $\{\phi_i\}$.
Therefore,
\begin{align*}
   \sum_{k=1}^{r+3}m_{1k}m_{kj}&=\sum_{k=1}^{r+3}\sum_{A\in H_2(X_r, \mathbb{Z})}\<c_1, \phi^1, \phi_k\>_A\sum_{A'\in H_2(X_r, \mathbb{Z})}\<c_1, \phi^k, \phi_j\>_{A'}\\
         &=\sum_{B\in H_2(X_r, \mathbb{Z})}\sum_{A+A'=B; A, A'\in H_2(X_r, \mathbb{Z})}\sum_{k=1}^{r+3}\<c_1, \phi^1, \phi_k\>_A \<c_1, \phi^k, \phi_j\>_{A'}\\
          &=\sum_{B\in H_2(X_r, \mathbb{Z})}\sum_{A+A'=B; A, A'\in H_2(X_r, \mathbb{Z})}\sum_{k=1}^{r+3}\<c_1, \phi^1, \phi^k\>_A \<c_1, \phi_k, \phi_j\>_{A'}\\
          &=\sum_{k=1}^{r+3}\sum_{A\in H_2(X_r, \mathbb{Z})}\<c_1, \phi^1, \phi^k\>_A\sum_{A'\in H_2(X_r, \mathbb{Z})}\<c_1, \phi_k, \phi_j\>_{A'}\\
          &= \sum_{A\in H_2(X_r, \mathbb{Z})}\<c_1, [pt], [pt]\>_A\sum_{A'\in H_2(X_r, \mathbb{Z})}\<c_1, \mathbbm1, E_{j-2}\>_{A'}\\
           &\qquad+ \sum_{A\in H_2(X_r, \mathbb{Z})}\<c_1, [pt], {H\over 3}\>_A\sum_{A'\in H_2(X_r, \mathbb{Z})}\<c_1, c_1, E_{j-2}\>_{A'}\\
           &\qquad+\sum_{k=3}^{r+2}\sum_{A\in H_2(X_r, \mathbb{Z})}\<c_1, [pt], {H\over 3}-E_{k-2}\>_A\sum_{A'\in H_2(X_r, \mathbb{Z})}\<c_1, E_{k-2}, E_{j-2}\>_{A'}\\
           &\qquad+\sum_{A\in H_2(X_r, \mathbb{Z})}\<c_1, [pt], \mathbbm1\>_A\sum_{A'\in H_2(X_r, \mathbb{Z})}\<c_1, [pt], E_{j-2}\>_{A'}.
\end{align*}
The above expressions make sense, as there are only finitely many nonzero terms in each summation. The third equality holds by a     change of bases. By Theorem \ref{vanishingproperty}, we have $\sum_A \<c_1, [pt], {H\over 3}-E_{k-2}\>_A=0$ for any $3\leq k\leq r+2$ (as ${H\over 3}-E_{k-2}\in \mbox{Prim}_r$ and $c_1, [pt]\in \mbox{Amb}_r$).
 Thus the last equality produces a positive number, due to the enumerative meaning of $3$-pointed genus zero Gromov-Witten invariants together with the non-negativity of the pairing  $\int_A E_{i}$ between  exceptional divisor classes and effective curve classes.

  By WDVV equations, we have the following equality.
  \begin{align*}
   \sum_{k=1}^{r+3}m_{2k}m_{kj}&=\sum_{k=1}^{r+3}\sum_{A\in H_2(X_r, \mathbb{Z})}\<c_1, \phi^2, \phi_k\>_A\sum_{A'\in H_2(X_r, \mathbb{Z})}\<c_1, \phi^k, \phi_j\>_{A'}\\
        &=\sum_{k=1}^{r+3}\sum_{A\in H_2(X_r, \mathbb{Z})}\<c_1, c_1, \phi_k\>_A\sum_{A'\in H_2(X_r, \mathbb{Z})}\<\phi^2, \phi^k, \phi_j\>_{A'}\\\end{align*}
   By the same arguments as above, we conclude $ \sum_{k=1}^{r+3}m_{2k}m_{kj}>0$ for $3\leq j\leq r+2$.
     \end{proof}
\begin{remark}
 Since $\sum_{k=1}^{r+3}m_{2k}m_{kj}=m_{2j}(m_{22}+d_r)+m_{21}m_{1j}+m_{2, r+3}$, we can also conclude the second inequality by   easily checking $m_{22}+d_r\geq 0$ for $r>4$ and simple calculation for $r=4$.
\end{remark}

\begin{lemma}\label{leminqm22}
   $m_{22}^{(2)}>d_r^2$.
\end{lemma}
\begin{proof}
  We have $m_{22}^{(2)}>m_{12}+m_{22}^2$.
  For $4\leq r\leq 7$, we have   $m_{12}=\sum_{\int_A=2} 4\<[pt]\>_A\geq C_{r}^14\<[pt]\>_{H-E_1}+ C_{r}^4 4\<[pt]\>_{2H-E_1-E_2-E_3-E_4}=4r+{1\over 6}r(r-1)(r-2)(r-3)>d_r^2$.
 For $r=8$,
     we have $m_{22}=\sum_{\int_A c_1=1}\int_A {H\over 3}\geq \sum_{1\leq i_1<i_2<i_3\leq 8}\int_{4H-\sum_{j=1}^8E_j- E_{i_1}-E_{i_2}-E_{i_3}} {H\over 3}={4\over 3}C_{8}^3>60=|d_8|$.
\end{proof}

\begin{proof}[Proof of Proposition \ref{propMatrixdelPezzo}]
  The statement is a direct consequence of the combination of the lemmas and Corollary \ref{corcondone} in this subsection.
\end{proof}

\section{Gamma conjecture I for del Pezzo surfaces}

In this section, we prove Gamma conjecture I for del Pezzo surfaces, by using the mirror techniques proposed by Galkin and Iritani \cite{GaIr} together with the study of Gamma conjecture I for certain weighted projective spaces.

\subsection{Toric del Pezzo surfaces}
 Let $X$ be an $n$-dimensional  toric Fano manifold $X$. In the context of mirror symmetry, the Landau-Ginzburg potential $f$ mirror to $X$ is
  a  Laurent polynomial \cite{Gi1,Gi2} of the form
\ban
f: (\bC^*)^n\rightarrow \bC;\, \mathbf{z} \mapsto f(\mathbf{z})=\mathbf{z}^{b_1}+\cdots+\mathbf{z}^{b_n},
\nan
where $b_1, \cdots, b_n\in \mathbb{Z}^n$ are primitive generators of the $1$-dimensional cones of the fan of $X$.
 As one remarkable property, the small quantum cohomology $QH^*(X)$ is isomorphic to the Jacobian ring $\mbox{Jac}(f)$ as algebras.
 The restriction $f|_{(\bR_{>0})^n}$ is a real function on $(\bR_{>0})^n$ that admits a global minimum at a unique point $\mathbf{z}_{con}\in(\bR_{>0})^n$ \cite{Gal,GaIr}; such point $\mathbf{z}_{con}\in(\bR_{>0})^n$ is called the {\it conifold point} of $f$.
 \begin{prop}\label{toricgammaconjI}(\cite[Theorem 6.3]{GaIr})
Suppose that $X$ is a toric Fano manifold satisfying the $B$-model analogue of Property $\cO$, namely for $T_{con}:=f(\mathbf{z}_{con})$,
\begin{enumerate}
\item every critical value $u$ of $f$ satisfies $|u|\leq T_{con}$;
\item $\mathbf{z}_{con}$ is the unique critical point of $f$ contained in $f^{-1}(T_{con})$.
\end{enumerate}
 Then $X$ satisfies Gamma conjecture I.
\end{prop}
\noindent We remark that the proof \cite{GaIr} of the above proposition uses the integral representation of the   central charge (see \cite{Hoso} and references therein for the notion of central charge).

There have been lots of studies on mirror symmetry for toric del Pezzo surfaces, namely for $X_r$ with $1\leq r\leq 3$. The on-shelf Landau-Ginzburg potential $f_r$  mirror to $X_r$ can be read off for instance from \cite[Example 2.3]{Jer}. Precisely, we have
$$f_1= z_1+z_2+{1\over z_1}+{1\over z_1 z_2}; f_2= z_1+z_2+{1\over z_1}+{1\over z_1 z_2}+{1\over z_2}; f_3= z_1+z_2+{1\over z_1}+{1\over z_1 z_2}+{1\over z_2} +z_1 z_2.$$
Therefore, we could have been done by easily verifying that these functions satisfy the hypotheses in Proposition \ref{toricgammaconjI} due to Galkin and Iritani. Nevertheless, we can also restrict to the study   of quantum cohomology by using the following consequence.
  \begin{cor}\label{toricFanoindexone}
Let $X$ be an $n$-dimensional toric Fano manifold. If $\tspec(\hat c_1)\cap \mathbb{R}_{>0}=\{\rho\}$ and   the multiplicity of $\rho$ is one, then  Gamma conjecture I holds for $X$.
\end{cor}
\begin{proof}
 It has been proved in \cite{Auroux,Iri} that   the set $\tspec(\hat c_1)$ of eigenvalues of   $\hat c_1$ coincides with the set of critical values of the   potential $f$ mirror to   $X$, and that the  multiplicities also coincide. (This is also a general expectation in mirror symmetry for   Fano manifolds.)     Since $\rho\in \tspec(\hat c_1)$, the condition (1) in Proposition \ref{toricgammaconjI} holds. Since $f|_{\mathbb{R}_{>0}^n}$ is a real function with positive real values and $f$ is holomorphic, any critical point $\mathbf{x}\in \mathbb{R}_{>0}^n$ of $f|_{\mathbb{R}_{>0}^n}$ is also a critical point of $f$ via the natural inclusion $\mathbb{R}_{>0}^n\subset (\mathbb{C}^*)^n$. Consequently, $f(\mathbf{x})\in \tspec(\hat c_1)\cap \mathbb{R}_{>0}$. It follows that $\rho=T_{con}$ and
 $f^{-1}(T_{con})$ is a single point set since the multiplicity of $\rho$ is one. Hence, the statement follows by Proposition \ref{toricgammaconjI}.
\end{proof}
By calculating the eigenvalues of the matrices $M_r (1\leq r\leq 3)$ in section 2.2, we can see that the hypotheses in the above corollary hold for toric del Pezzo surfaces, where we include the known cases $\mathbb{P}^1\times \mathbb{P}^1$ and $\mathbb{P}^2$. Hence, we have the following.
\begin{prop}\label{X123}
Toric del Pezzo surfaces satisfy Gamma conjecture I.
\end{prop}
\iffalse
\begin{proof}
The $B$-model analogue of Property $\cO$ is easy to check for $\bP^2$ and $\bP^1\times\bP^1$. For $X_r(1\leq r\leq3)$, we use results from Section \ref{QHdelPezzo} to find the characteristic polynomial of $\hat c_1$:
\ban
&&X_1: x^4+x^3-8x^2-36x-11;\\
&&X_2: x^5+3x^4-15x^3-78x^2-104x-43;\\
&&X_3: x^6+6x^5-15x^4-208x^3-648x^2-864x-432.
\nan
Now one can use MAPLE to check that $X_r(1\leq r\leq3)$ satisfies the conditions in Corollary \ref{toricFanoindexone}.
\end{proof}
\fi

\subsection{Non-toric cases} \label{GammaconjectureIfornontoricdelPezzosurfaces}
Non-toric del Pezzo surfaces can be described as complete intersections in nice ambient spaces of Picard rank one. As proposed   in  \cite[Section 8]{GaIr}, we shall use  the quantum Lefschetz principle to prove Gamma conjecture I in these cases.

Let us start with the precise quantum Lefschetz principle in Proposition \ref{QLP} as well as its proof  following \cite[Theorem 8.3]{GaIr}. We give the details here since some additional work has to been done in the case of del Pezzo surfaces, which concerns about the notion of  primitive part as given below.

\begin{defn}\label{ambprim}
Let $\iota:Y\hookrightarrow X$ be an embedding of a smooth projective variety $Y$ as a hypersurface into a (possibly singular) projective variety $X$. Then with respect to $\iota$, the \textbf{ambient part} of $H^\bullet(Y)$ is $\iota^*H^\bullet(X)$, and the \textbf{primitive part} of $H^\bullet(Y)$ is the orthogonal complement of $\iota^*H^\bullet(X)$ in $H^\bullet(Y)$ with respect to the Poincar\'e pairing.
\end{defn}

\begin{prop}\label{QLP}
Let $X$ be a Fano manifold of index $r_X\geq2$ and write $-K_X=r_Xh$. Let $\iota:Y\hookrightarrow X$ be a Fano hypersurface in the linear system $|ah|$ with $0<a<r$. Assume that
\ban\label{limitformulaforX}
\hat\Gamma_X\varpropto\lim\limits_{t\rightarrow+\infty}t^{\frac{\dim X}2}e^{-C_Xt}J_X(t)
\nan
for some constant $C_X$, and that the primitive part of $J_Y(t)$ vanishes. Then
\ba\label{limitformulaforY}
\hat\Gamma_Y\varpropto\lim\limits_{t\rightarrow+\infty}t^{\frac{\dim Y}2}e^{-C_Yt}J_Y(t)
\na
for some constant $C_Y$. In particular if $Y$ satisfies Property $\mathcal{O}$, then $Y$ satisfies Gamma conjecture I.
\end{prop}
\begin{proof}
Write
%\ban
$J_X(t)=e^{r_Xh\log t}\suml_{n=0}^\infty J_{r_Xn}t^{r_Xn},\textrm{ with }J_{r_Xn}\in H^\bullet(X).$
%\nan

Since the primitive part of $J_Y(t)$ vanishes, $J_Y(t)$ takes values in $\iota^*H^\bullet(X)$. It follows from the quantum Lefschetz principle \cite{Lee, CoGi} that
\ba\label{QLPforY}
J_Y(t)=e^{(r_X-a)h\log t-C_0t}\suml_{n=0}^\infty\frac{\Gamma(1+ah+an)}{\Gamma(1+ah)}(\iota^*J_{r_Xn})t^{(r_X-a)n},
\na
where $C_0$ is some constant determined by $Y$. Set
$$
\tilde J(t)=t^{\frac{\dim X}2}e^{-C_Xt}J_X(t),\quad T_0=\max\limits_{q\geq0}\{-q+C_Xq^{{a\over r_X}}\},\quad\mbox{and}\quad
C_Y=T_0-C_0.
$$
We have
\ban
t^{\frac{\dim Y}2}e^{-C_Yt}J_Y(t)=\frac{\sqrt t}{\Gamma(1+ah)}\int_0^\infty q^{-\frac{a\dim X}{2r_X}}e^{-(q-C_Xq^{ a\over r_X}+T_0)t}\tilde J(tq^{a\over r_X})dq.
\nan
Using the stationary phase approximation, we conclude  that RHS of \eqref{limitformulaforY} is proportional to $\frac{\iota^*\hat\Gamma_X}{\Gamma(1+ah)}$. Now the first required result follows from the  equality
$
\hat\Gamma_Y=\frac{\iota^*\hat\Gamma_X}{\Gamma(1+ah)},
$ as obtained from the adjunction formula.

If $Y$ satisfies Property $\mathcal{O}$, then  $Y$   satisfies Gamma conjecture I by Proposition \ref{propJexpansion}.
\end{proof}
\noindent We remark that the above proof is more like an outline, and refer to \cite[Theorem 8.3]{GaIr} for the details of  estimations on  the asymptotics skipped here. The assumption on $X$ above is  slightly weaker than the requirement of $X$ satisfying Gamma conjecture I.
 The assumption on the vanishing of the primitive part of $J_Y(t)$, which was  not explicitly mentioned in \cite[Section 8]{GaIr}, can be guaranteed
 whenever $\dim Y\geq 3$. This is due to the next property following from (the proof of) \cite[Lemma 1]{LePa}.
\begin{lemma}\label{primitivevanishing}
Let $X$ be a Fano manifold of index $r\geq2$. Let $Y$ be a Fano hypersurface of $X$  with $c_1(Y)\varpropto -K_X$.  If $\dim Y\geq3$, then the primitive part of $J_Y(t)$ vanishes.
\end{lemma}

\subsubsection{Case $X_r (4\leq r\leq6)$}

%%Now we discuss the cases  $X_r(4\leq r\leq 6)$.

\begin{prop}\label{X456}
For each $4\leq r\leq 6$, $X_r$ satisfies Gamma conjecture I.
\end{prop}
\begin{proof}
%%We use the embedding of $X_r(4\leq r\leq 6)$ into Grassmannians as described in Section \ref{geometricinterpretation} to prove this proposition.

%For $r=3$, $X_3$ is embedded into $\bP^2\times\bP^2$, and we can check that it satisfies the $B$-model analogue of Property $\cO$, which implies that it satisfies Gamma conjecture I from Theorem \ref{toricgammaconjI}.
For $r=4$, we consider the embedding of $X_4$ as complete intersections in complex Grassmannian $Gr(2,5)$ \cite{Cord}. More precisely,
\ban
X_4=Gr(2,5)\cap H_1\cap H_2\cap H_3\cap H_4,
\nan
where $Gr(2,5)$ is viewed as a subvariety in  $\bP^9$ via the Pl\"ucker embedding and  $H_i$'s are     hyperplanes in $\bP^9$ in general position, so that the above intersection  makes sense. Denote
\ban
Y_k:=Gr(2,5)\bigcap\bigcap\limits_{i=1}^kH_i,\quad0\leq k\leq4.
\nan
 With respect to the embedding $Y_k\subset Y_{k-1}$, the primitive part of $J_{Y_k}(t)(1\leq k\leq3)$ vanishes by Lemma \ref{primitivevanishing}, and the primitive part of $J_{Y_4}(t)$ vanishes by Theorem \ref{vanishingproperty}. The Picard rank of $Y_0=Gr(2, 5)$ equals one, and $Y_0$ satisfies  Gamma conjecture I by \cite[Theorem 6.1.1]{GGI}.
 Thus $Y_0$ satisfies the hypotheses on the ambient space in Proposition \ref{QLP}. By applying Proposition \ref{QLP} and using induction on $k$,  we conclude that for any $1\leq k\leq 4$,
\ban
\hat\Gamma_{Y_k}\varpropto\lim\limits_{t\rightarrow+\infty}t^{\frac{\dim Y_k}2}e^{-C_{Y_k}t}J_{Y_k}(t),\textrm{ for some constant }C_{Y_k}.
\nan
Since $Y_4=X_4$ satisfies Property $\cO$, it follows   that $X_4$ satisfies Gamma conjecture I.

The arguments for cases $r=5$ and $r=6$ are similar.
\end{proof}

\subsubsection{Case $X_r (7\leq r\leq8)$}

The weighted projective space $\mathscr X:=\bP(1,w_1,\cdots,w_N)$ is the quotient stack $[(\bC^{N+1}\setminus\{0\})/\bC^*]$  with $\bC^*$-action   with weights $-1,-w_1,\cdots,-w_N$. The del Pezzo surfaces $X_7, X_8$ are   smooth hypersurfaces in $\mathcal X$  respectively in the special cases $(w_1, w_2, w_3)=(1,1,2)$ and $(1,2,3)$.
In order to show Gamma conjecture I for $X_7, X_8$, we will first study that for   $\mathscr X$, and then apply the corresponding  quantum Lefschetz principle. We refer our readers to \cite{CCLT, Iri} for basic materials of orbifold Gromov-Witten theory  of weighted projective spaces.

To ease notations, in the rest of this section, we denote
$$
r_{\mathscr X}:=1+w_1+\cdots+w_N,\quad c=r_{\mathscr X}(\prodl_{i=1}^Nw_i^{-w_i})^{1\over r_{\mathscr X}}\quad\mbox{and}\quad  h=c_1(\cO(1))\in H^\bullet(\mathscr X).$$
Notice that the first Chern class of $\mathscr X$ is given by $c_1(\mathscr X)=r_{\mathscr X} h$.

Following \cite[section 1]{CCLT}, we let $F:=\{\frac k{w_i}|1\leq i\leq N,0\leq k<w_i\}$. For $f\in F$, we let $\mathscr X_f$ be the locus of points of $\mathscr X$ with isotropic group containing $e^{2\pi\sqrt{-1}f}$.
The Chen-Ruan orbifold cohomology group of $\mathscr X$ (denoted also by $H_{\rm CR}^\bullet(\mathscr X)$) is given by
\ban
H^\bullet_{orb}(\mathscr X):=H^\bullet(I\mathscr X)=H^\bullet(\mathscr X)\oplus\bigoplus_{f\in F\setminus\{0\}}H^\bullet(\mathscr X_f), \mbox{ where } I\mathscr X:=\bigsqcup_{f\in F}\mathscr X_f.
\nan
Here $I\mathscr X$ is called the inertia stack of $\mathscr X$, and we notice $\mathscr X_0=\mathscr X$. The bullet '$\bullet$' means that we only consider classes with even topological degree.

The quantum connection $\nabla$ on the trivial $H^\bullet_{orb}(\mathscr X)$-bundle over $\bP^1$ is given by
\ban
\nabla_{z\partial_z}=z\partial_z-\frac1z(c_1(\mathscr X)\bullet)+\mu,
\nan
where $\bullet$ is the orbifold small quantum product of $\mathscr X$ with all Novikov variables setting to $1$, and $\mu$ is the Hodge grading operator respecting the Chen-Ruan degree of classes in $H^\bullet_{orb}(\mathscr X)$.  The quantum connection is a meromorphic connection, which is logarithmic at $z=\infty$ and irregular at $z=0$. Consider the holomorphic function
\ban
S: \mathbb{P}^1\setminus\{0\}&\to&\mbox{End}(H^\bullet_{orb}(\mathscr X))
\nan
defined by
\ban
(S(z)(\alpha),\beta)^{orb}_\mathscr X=(\alpha,\beta)^{orb}_\mathscr X+\suml_{m\geq0}\frac1{z^{m+1}}\suml_{d}(-1)^{m+1}\<\alpha\psi^m,\beta\>^\mathscr X_d.
\nan
Then from \cite[Section 2.3]{Iri}, the space of flat sections can be identified with the Chen-Ruan orbifold cohomology group $H^\bullet_{orb}(\mathscr X)$ via the fundamental solution $S(z)z^{-\mu}z^{c_1(\mathscr X)}$. In particular,  we have the following isomorphism to the space of $\nabla$-flat sections over $\bR_{>0}$.
\ban
\Phi:H^\bullet_{orb}(\mathscr X)&\xrightarrow{\cong}&\{s:\bR_{>0}\to H^\bullet_{orb}(\mathscr X):\nabla s=0\};\\
\alpha&\mapsto&(2\pi)^{-\frac N2}S(z)z^{-\mu}z^{c_1(\mathscr X)}\alpha.
\nan

It follows from \cite[Theorem 1.1]{CCLT}  that the characteristic polynomial of $(c_1(\mathscr X)\bullet)$ is $\lambda^r-c^r$. By parallel discussions to that in \cite[Section 3.2]{GGI}, we have

\begin{prop}[Analogue to Proposition 3.3.1 of \cite{GGI}]\label{smllestasymptotic}
Let
\begin{align*}
   \mathscr A :=&\,\,\{s:\bR_{>0}\to H^\bullet_{orb}(\mathscr X)|\nabla s=0,
 ||e^{\frac cz}s(z)||=O(z^{m})\textrm{ as }z\to+0 \mbox{ for some }m\in\bZ_{\geq0}\},\\
E_c :=&\,\,\textrm{eigenspace of }(c_1(\mathscr X)\bullet)\textrm{ in }H_{orb}^\bullet(\mathscr X)\textrm{ with eigenvalue }c.
\end{align*}
%\ban
%\mathscr A :=&\{s:\bR_{>0}\to H^\bullet_{orb}(\mathscr X)|\nabla s=0,    ||e^{\frac cz}s(z)||=O(z^{m})\textrm{ as }z\to+0 \mbox{ for some }m\in\bZ_{\geq0}\},\\
%%\mathscr A&:=&\{s:\bR_{>0}\to H^\bullet_{orb}(\mathscr X)|\nabla s=0,\textrm{ and there exists }m\in\bZ_{\geq0}\textrm{ such that}\\
%%&&\qquad\qquad\qquad\qquad||e^{\frac cz}s(z)||=O(z^{m})\textrm{ as }z\to+0\},\\
%E_c :=&\textrm{eigenspace of }(c_1(\mathscr X)\bullet)\textrm{ in }H_{orb}^\bullet(\mathscr X)\textrm{ with eigenvalue }c.
%\nan
Then both $\mathscr A$ and $E_c$ are one-dimensional complex vector spaces, and they are isomorphic to each other via the map
${\displaystyle \mathscr A\to E_c\mbox{ defined by }s \mapsto \lim\limits_{z\to+0}e^{\frac cz}s(z)}$
%%\ban \mathscr A&\to&E_c\\ s&\mapsto&\lim\limits_{z\to+0}e^{\frac cz}s(z). \nan
\end{prop}

\begin{defn}
A \textbf{principal asymptotic class} of $\mathscr X$ is a nonzero class $A_\mathscr X\in H_{orb}^\bullet(\mathscr X)$ such that $\Phi(A_\mathscr X)\in\mathscr A$.
\end{defn}

By Proposition \ref{smllestasymptotic}, the principal asymptotic class  of $\mathscr X$ is unique up to a nonzero scalar. Actually, it is  given  by the Gamma class of $\mathscr X$ as  below. The Gamma class  could be defined for an almost complex orbifold,  and lives in the orbifold cohomology group \cite[(23)]{Iri}. The Gamma class of $\mathscr X$ is of the form
\ban
\hat\Gamma_\mathscr X=\Gamma(1+h)\prodl_{i=1}^N\Gamma(1+w_ih)+\textrm{ terms in }\bigoplus_{f\in F\setminus\{0\}}H^\bullet(\mathscr X_f).
\nan
The following  can be viewed as one (equivalent) version of Gamma conjecture I for $\mathscr X$.
\begin{prop}\label{gammaconjIforweightedprojectivespaces}
$\hat\Gamma_\mathscr X$ is a principal asymptotic class of $\mathscr X$.
\end{prop}
\begin{proof}
Consider the Laurant polynomial
\ban
f=x_1+\cdots+x_N+\frac1{x_1^{w_1}\cdots x_N^{w_N}},\quad (x_1,\cdots,x_N)\in(\bC^*)^N.
\nan
The next mirror identity follows from the argument in \cite[Section 4.3.1]{Iri}:
\ba\label{integralrepresentationbyiritani}
\bigg(\phi,S(z)z^{-\mu}z^{rh}\hat\Gamma_\mathscr X\bigg)_\mathscr X^{orb}=z^{-\frac N2}\int_{(\bR_{>0})^N}e^{-\frac{f(x)}z}\varphi(x,z)\frac{dx_1\cdots dx_N}{x_1\cdots x_N},
\na
where $z>0$, $\phi\in H^\bullet_{orb}(\mathscr X)$, and $\varphi(x,z)\in\bC[x_1^{\pm},\cdots,x_N^{\pm},z]$ is such that the class of the integrand on RHS corresponds to $\phi$ under the mirror isomorphism in \cite{Iri}. To study the oscillatory integral on RHS, we find all the critical points $p_k$ of $f$ by direct calculations:
\ban
p_k:=\bigg(\frac{w_1\xi^k}{\sqrt[r_{\mathscr X}]{(\prodl_{j=1}^Nw_j^{w_j})}},\cdots,\frac{w_N\xi^k}{\sqrt[r_{\mathscr X}]{(\prodl_{j=1}^Nw_j^{w_j})}}\bigg),\quad k=0,1,\cdots,r_{\mathscr X}-1,
\nan
where $\xi$ is a primitive $r_{\mathscr X}$th root of unity. Moreover, we can check that $f|_{(\bR_{>0})^N}$ admits a global minimum at the unique point $p_0$  with $c=f(p_0)$, satisfying  the following    properties:
\begin{enumerate}
\item every critical value $u$ of $f$ satisfies $|u|\leq c$;
\item $p_0$ is the unique critical point of $f$ contained in $f^{-1}(c)$.
\end{enumerate}
Therefore by the stationary phase approximation, we have
\ban
||e^{\frac{c}{z}}S(z)z^{-\mu}z^{r_{\mathscr X}h}\hat\Gamma_\mathscr X||=O(1),\quad z\rightarrow+0.
\nan
This implies that the flat section $S(z)z^{-\mu}z^{r_{\mathscr X}h}\hat\Gamma_\mathscr X$ is in $\mathscr A$.
\end{proof}

As for Fano manifolds, we can also interpret the principal asymptotic class of $\mathscr X$ in terms of  Givental's $J$-function of $\mathscr X$  defined by $J_\mathscr X(t)=z^{\frac N2}\big(S(z)z^{-\mu}z^{c_1(\mathscr X)}\big)^{-1}\mathbbm 1$ with $t={1\over z}$.

\begin{thm}\label{proportionforweightedprojectivespaces}
 $\mathscr X$ satisfies Gamma conjecture I, namely $\hat\Gamma_{\mathscr X}\varpropto\lim\limits_{t\rightarrow+\infty}t^{\frac{N}2}e^{-ct}J_{\mathscr X}(t)$.
\end{thm}

\begin{proof}
Denote by $\<\cdot,\cdot\>$ the natural pairing  between $H_{orb}^\bullet(\mathscr X)^\vee$ and $H_{orb}^\bullet(\mathscr X)$, where $H_{orb}^\bullet(\mathscr X)^\vee$ is the dual space of $H_{orb}^\bullet(\mathscr X)$.
The dual connection $\nabla^\vee$ of $\nabla$ is a meromorphic connection on the trivial   $H_{orb}^\bullet(\mathscr X)^\vee$-bundle over $\bP^1$ given by
\ban
\nabla_{z\partial_z}^\vee=z\partial_z+\frac1z(c_1(\mathscr X)\bullet)^\vee-\mu^\vee,
\nan
where $(c_1(\mathscr X)\bullet)^\vee$ and $\mu^\vee$ are dual maps of $(c_1(\mathscr X)\bullet)$ and $\mu$ respectively. Here $\nabla^\vee$ is dual to $\nabla$ in the sense that
\ba\label{dual}
d\<f(z),s(z)\>=\<\nabla^\vee f(z),s(z)\>+\<f(z),\nabla s(z)\>.
\na
For the space of $\nabla^\vee$-flat sections over $\bR_{>0}$, we have the following isomorphism:
\ban
\Phi^\vee:H^\bullet_{orb}(\mathscr X)^\vee&\xrightarrow{\cong}&\{f:\bR_{>0}\to H^\bullet_{orb}(\mathscr X)^\vee:\nabla^\vee f=0\};\\
\alpha&\mapsto&(2\pi)^{\frac N2}((S(z)z^{-\mu}z^{c_1(\mathscr X)})^{-1})^\vee\alpha.
\nan

By  Proposition \ref{smllestasymptotic}, we obtain a nonzero element $\phi_0$ in $E_c$ defined by
\ban
\phi_0:=\lim\limits_{z\to+0}e^{\frac cz}\Phi(\hat\Gamma_\mathscr X)(z)\in E_c.
\nan
 Let $E_c^\vee$ be the eigenspace of $(c_1(\mathscr X)\bullet)^\vee$ in $H_{orb}^\bullet(\mathscr X)^\vee$ with eigenvalue $c$. Then $E_c^\vee$ can be viewed naturally as the dual space of $E_c$, since   the pairing $\<\cdot,\cdot\>|_{E_c^\vee\times E_c}$ is nondegenerate by direct verification. Let $\phi_0^\vee\in E_c^\vee$ be the dual of $\phi_0$ in the sense that
\ban
\<\phi_0^\vee,\phi_0\>=1.
\nan
For any $\alpha\in H^\bullet_{orb}(\mathscr X)^\vee$, we   obtain the following  formula (cf. \cite[Proposition 3.6.2]{GGI}):
\ba\label{leadingtermofdualflatsection}
\lim\limits_{z\to+0}e^{-\frac cz}\Phi^\vee(\hat\Gamma_\mathscr X)(z)=\<\alpha,\hat\Gamma_\mathscr X\>\phi_0^\vee,
\na
by using the parallel disccusions to that  in \cite[Section 3.5, Section 3.6]{GGI}. Note that
\ban
t^{\frac{N}2}e^{-ct}J_{\mathscr X}(t)=(2\pi)^{-\frac N2}e^{-\frac cz}\Phi^{-1}(\mathbbm1)(z)\textrm{ with }t=\frac 1z.
\nan
Let $C=\<\phi_0^\vee,\mathbbm1\>$. For any $\alpha\in H^\bullet_{orb}(\mathscr X)^\vee$, we have
\ban
 \lim\limits_{t\to+\infty}\<\alpha,t^{\frac{N}2}e^{-ct}J_{\mathscr X}(t)\>
&=&\lim\limits_{z\to+0}\<\alpha,(2\pi)^{-\frac N2}e^{-\frac cz}\Phi^{-1}(\mathbbm1)(z)\>\\
&=&\lim\limits_{z\to+0}\<\Phi^\vee(\alpha)(z),(2\pi)^{-\frac N2}e^{-\frac cz}\mathbbm1\>\\
&=&\lim\limits_{z\to+0}\<e^{-\frac cz}\Phi^\vee(\alpha)(z),(2\pi)^{-\frac N2}\mathbbm1\>\\
&=&\<\alpha,C\cdot (2\pi)^{-\frac N2}\hat\Gamma_\mathscr X\>.
\nan
  Here  the second (resp. third) equality follows from   \eqref{dual} (resp. \eqref{leadingtermofdualflatsection}). This implies that
\ban
\lim\limits_{t\rightarrow+\infty}t^{\frac{N}2}e^{-ct}J_{\mathscr X}(t)=C\cdot(2\pi)^{-\frac N2}\hat\Gamma_\mathscr X.
\nan
So we only need to show that $C\neq0$.

Note that $H':=\mbox{Image}(c-(c_1(\mathscr X)\bullet))$ is a complementary subspace of $E_c$ in $H_{orb}^\bullet(\mathscr X)$. Therefore there exists  $\gamma\in H_{orb}^\bullet(\mathscr X)$ such that
\ban
\mathbbm1=C'\phi_0+(c\gamma-c_1(\mathscr X)\bullet\gamma)\textrm{ for some }C'\in\bC.
\nan
Moreover,
\ban
C&=&\<\phi_0^\vee,\mathbbm1\>\\
&=&C'+\<\phi_0^\vee,c\gamma\>-\<\phi_0^\vee,c_1(\mathscr X)\bullet\gamma\>\\
&=&C'+\<\phi_0^\vee,c\gamma\>-\<(c_1(\mathscr X)\bullet)^\vee\phi_0^\vee,\gamma\>\\
&=&C'+\<\phi_0^\vee,c\gamma\>-\<c\phi_0^\vee,\gamma\>\\
&=&C'.
\nan
Therefore,
\ban
\phi_0&=&C\phi_0\bullet\phi_0+\phi_0\bullet(c\gamma-c_1(\mathscr X)\bullet\gamma)\\
&=&C\phi_0\bullet\phi_0+c\phi_0\bullet\gamma-c_1(\mathscr X)\bullet\phi_0\bullet\gamma\\
&=&C\phi_0\bullet\phi_0+c\phi_0\bullet\gamma-c\phi_0\bullet\gamma\\
&=&C\phi_0\bullet\phi_0.
\nan
Since $\phi_0\neq 0$, it follows that   that $C\neq0$. This finishes the proof.
\end{proof}

Now we discuss   hypersurfaces in $\mathscr X$ by   the quantum Lefschetz principle for orbifolds.

\begin{prop}\label{QLPinweightedprojectivespace}
Let $Y$ be a smooth Fano variety given as a hypersurface in $\mathscr X$ defined by a section of $\cO(d)$ for a positive integer $d$. Assume that $1+w_1+\cdots+w_N-d>0$, that $w_i\mid d$ for $1\leq i\leq N$, and that the primitive part of $J_Y(t)$ vanishes. Then
\ban\label{GammeOrbi}
\hat\Gamma_Y\varpropto\lim\limits_{t\rightarrow+\infty}t^{\frac{\dim Y}2}e^{-C_Yt}J_Y(t)
\nan
for some constant $C_Y$.
\end{prop}
\begin{proof}
The argument is   similar to that for  \cite[Theorem 8.3]{GaIr} again.  Write
\ban
J_\mathscr X(t)=e^{r_{\mathscr X}h\log t}\suml_{n=0}^\infty J_{r_{\mathscr X}n}t^{r_{\mathscr X}n},\textrm{ with }J_{r_{\mathscr X}n}\in H_{orb}^\bullet(\mathscr X).
\nan
Since the primitive part of $J_Y(t)$ vanishes,  $J_Y(t)$ takes values in $\iota^*H^\bullet(\mathscr X)$. It follows from \cite[Corollary 1.9]{CCLT} and the discussion in the proof of \cite[Proposition D.9]{CCGK} that
\ba\label{QLPforYinweightedprojectivespaces}
J_Y(t)=e^{(r_{\mathscr X}-d)h\log t-C_0t}\suml_{n=0}^\infty\frac{\Gamma(1+dh+dn)}{\Gamma(1+dh)}(\iota^*pr^*J_{r_{\mathscr X}n})t^{(r_{\mathscr X}-d)n},
\na
where $C_0$ is some constant determined by $Y$, and $pr:H^\bullet_{orb}(\mathscr X)\rightarrow H^\bullet(\mathscr X)$ is the natural projection. Set
$$\tilde J(t)=t^{\frac{N}2}e^{-ct}J_\mathscr X(t),\quad
T_0=\max\limits_{q\geq0}\{-q+cq^{{d\over r_{\mathscr X}}}\}\quad\mbox{and}\quad C_Y=T_0-C_0.
$$
 We have
\ban
t^{\frac{\dim Y}2}e^{-C_Yt}J_Y(t)=\frac{\sqrt t}{\Gamma(1+dh)}\int_0^\infty q^{-\frac{dN}{2r_{\mathscr X}}}e^{-(q-cq^{{d\over r_{\mathscr X}}}+T_0)t}\tilde J(tq^{{ d\over r_{\mathscr X}}})dq.
\nan
Again using the stationary phase approximation as in the proof of \cite[Theorem 8.3]{GaIr}, we conclude  that RHS of \eqref{GammeOrbi} is proportional to $\frac{\iota^*pr^*\hat\Gamma_\mathscr X}{\Gamma(1+dh)}$. Now the required result follows from the next equality:
\ban
\hat\Gamma_Y=\frac{\iota^*pr^*\hat\Gamma_\mathscr X}{\Gamma(1+dh)}=\frac{\Gamma(1+h)\prodl_{i=1}^N\Gamma(1+w_ih)}{\Gamma(1+dh)}.
\nan
which is obtained  by the adjunction formula.
\end{proof}

\begin{cor}\label{conjOimpliesgammaconjIinweightedprojectivespaces}
Let $Y$ be as in the assumptions of Proposition \ref{QLPinweightedprojectivespace}. If $Y$ satisfies Property $\cO$, then $Y$ satisfies Gamma conjecture I.
\end{cor}

%\begin{remark}
%One can also show that weighed projective spaces $\bP(w_0,w_1,\cdots,w_N)$ satisfies Gamma conjecture I, by following the approach of \cite[Section 5]{GGI}.
%\end{remark}

\begin{prop}\label{X78}
$X_7$ and $X_8$ satisfy Gamma conjecture I.
\end{prop}
\begin{proof}
The del Pezzo surface  $X_7$ can be realized as a smooth hypersuface in $\bP(1,1,1,2)$ defined by a section of $\cO(4)$. By Theorem \ref{vanishingproperty}, the primitive part of $J_{X_7}(t)$ vanishes. So Gamma conjecture I holds for $X_7$ by  Theorem \ref{proportionforweightedprojectivespaces}, Proposition  \ref{QLPinweightedprojectivespace} and Corollary \ref{conjOimpliesgammaconjIinweightedprojectivespaces}.

The proof for $X_8$ is similar.
\end{proof}

In a summary,   we achieve  Theorem \ref{thmGammconjI} by combining Propositions \ref{X123}, \ref{X456} and \ref{X78}.

\section*{Acknowledgements}

%% The first author is supported in part by................%
% by NRF-2007-341-C00006.
The authors would like to thank Kwok Wai Chan, Xiaowen Hu, Yuan-Pin Lee and Bong H. Lian   for  useful discussions, and to thank Yuri Prokhorov and Mao Sheng for the helpful discussions on interpreting toric del Pezzo surfaces as complete intersections in product of projective spaces.
 Hu is partially supported by NSFC Grants 11831017 and 11771460. Ke is partially supported by NSFC Grants 11831017, 11521101 and 11601534. Li  is partially supported by NSFC Grants 11822113, 11831017 and Guangdong Introducing Innovative and Enterpreneurial Teams No. 2017ZT07X355.

\end{document}